\documentclass[11pt,oneside]{amsart}
\usepackage{soul}
\usepackage[normalem]{ulem}
\usepackage{amssymb,amsmath,amsbsy,mathrsfs,eucal,amsfonts,eucal,amscd,exscale}
\usepackage{graphics}
\usepackage{verbatim}
\usepackage{appendix}
\usepackage{rotating}
\usepackage{enumerate}
\usepackage{amsthm}
\usepackage[]{algorithm2e}
\usepackage{xspace}
\usepackage{comment}
\usepackage{tikz,tikz-3dplot}
\usepackage{subfig}

\usepackage{color}
\definecolor{red}{rgb}{1,0,0}
\newcommand\red[1]{\textcolor{red}{#1}}
\definecolor{blue}{rgb}{0,0,1}

\usepackage{tikz}
\usetikzlibrary{arrows}

\numberwithin{equation}{section}
\numberwithin{table}{section}
\numberwithin{figure}{section}

\newcommand{\N}{\mathcal{N}}
\newcommand{\MA}{Monge-Amp\`ere }
\newcommand{\bld}[1]{\boldsymbol{#1}}
\newcommand{\vf}[1]{\hbox{\boldmath$#1$}}
\newcommand{\abs}[1]{| #1 |}

\newcommand{\dist}{\textrm{dist}}

\newcommand{\dm}{\Omega}

\newcommand{\bdry}{\partial \Omega}

\newcommand{\T}{{\mathcal{T}}}
\newcommand{\Th}{{\mathcal{T}_{h}}}

\newcommand{\Tk}{{\mathcal{T}_{k}}}

\newcommand{\Nh}{{\mathcal{N}_{h}}}

\newcommand{\Nk}{{\mathcal{N}_{k}}}
\newcommand{\NR}{{\mathcal{N}_{h/2}}}
\newcommand{\Nhi}{{\mathcal{N}_{h}^0}}
\newcommand{\Nhb}{{\mathcal{N}_{h}^{\partial}}}

\newcommand{\hull}{\textrm{hull}}
\newcommand{\jump}[1]{\,[\![#1]\!]}

\newcommand{\gradv}{\bld{\nabla}}

\newcommand{\norm}[1]{\|\,#1\,\|}
\newcommand{\inftynorm}[1]{\|\,#1\,\|_{L^{\infty}(\dm_h)}}

\newcommand{\stkout}[1]{\red{\ifmmode\text{\sout{\ensuremath{#1}}}\else\sout{#1}\fi}}

\newtheorem*{condition*}{Condition}

\newtheorem{theorem}{Theorem}[section]
\newtheorem{lemma}{Lemma}[section]
\newtheorem{prop}[lemma]{Proposition}
\newtheorem{corollary}[lemma]{Corollary}
\newtheorem{definition}{Definition}[section]
\newtheorem{example}{Example}[section]
\theoremstyle{remark}
\newtheorem{remark}{Remark}[section]

\title{ Pointwise Rates of Convergence for the Oliker-Prussner Method
for the Monge-Amp\`ere Equation }

\author{Ricardo H. Nochetto}
\address{Department of Mathematics and Institute for Physical
  Science and Technology, University of Maryland, College Park}
\curraddr{}
\email{rhn@math.umd.edu}
\thanks{Both authors were partially supported by NSF Grants
  DMS-1109325 and DMS-1411808.}

\author{Wujun Zhang}
\address{    Department of Mathematics, Rutgers University}
\curraddr{}
\email{wujun@math.rutgers.edu}
\thanks{The second author was also partially supported by the Brin
  Postdoctoral Fellowship of the University of Maryland
  and the start up fund of Rutgers University.}

\begin{document}
\maketitle
\keywords

\begin{abstract}
We study the Oliker-Prussner method exploiting its geometric nature.
We derive discrete stability and continuous dependence estimates in
the max-norm by using a discrete Alexandroff estimate and the
Brunn-Minkowski inequality. We show that the method is exact for all convex
quadratic polynomials provided the underlying set of nodes is
translation invariant within the domain; nodes still conform to the
domain boundary. This gives a suitable notion of operator consistency which,
combined with stability, leads to pointwise
rates of convergence for classical and non-classical
solutions of the \MA equation.
\end{abstract}
\pagestyle{myheadings}
\thispagestyle{plain}
\markboth{R.H. Nochetto and W. Zhang}{Rates of convergence for the \MA equation}

%
\section{Introduction}\label{S:introduction}
%
We consider the fully nonlinear \MA equation  
\begin{subequations}
\label{pde}
\begin{align}\label{pde_a}
\det D^2 u & = f  \qquad \text{in $\Omega$}
\\
\label{pde_b}
u & = g  \qquad\text{on $\partial \Omega$},
\end{align}
\end{subequations}
where $\Omega$ denotes a {\em uniformly convex} domain in
$\mathbb{R}^d$ ($d \ge 2$), $f \in C(\overline{\dm})$
satisfies $ 0 < \lambda_F \leq f(x) \leq \Lambda_F$ for all $x\in\Omega$,
and $g \in C(\partial{\dm})$.
Such an equation arises in differential geometry, optimal mass
transport and several fields of science and engineering, and
has received considerable attention in recent years.

In contrast to an extensive PDE literature \cite{Gutierrez01,
CaffarelliCabre95,CrandallIshiiLions92}, the numerical approximation
is still under development.
Convergence to the {\it viscosity solution} (see Definition
\ref{D:viscositysolution}), in the general framework of fully nonlinear
elliptic PDEs, is studied in the early works \cite{BarlesSouganidis91,
  KuoTrudinger90, KuoTrudinger92, KuoTrudinger00} and hinges on operator stability, consistency and monotonicity
properties. A. Oberman et al. designed
several finite difference methods (wide stencil schemes) within this framework
\cite{Oberman08, FroeseOberman11SIAM, FroeseOberman11JCP, BenamouFroeseOberman10}.
Benamou et. al. proposed recently a convergent finite difference
method, which reduces the stencil size,
and showed that the method is consistent for convex
polynomials \cite{BenamouCollinoMirebeau16}.

A geometric notion of generalized solution for the \MA equation is the so
called {\it Alexandroff solution} (see Definition \ref{D:Alexandroff}).
It hinges on the geometric interpretation of $\int_D\det D^2u$ as the measure
of the subdifferential $|\partial u(D)|$
of a convex function $u$ for any Borel set $D$. This notion of
solution is weaker than the viscosity solution.
V. Oliker and L. Prussner developed a discrete counterpart and proved
convergence of the ensuing geometric numerical method
\cite{OlikerPrussner88}. Even though this is a natural idea, the
scheme has defied analysis ever since its conception. The purpose of
this paper is to fill this gap upon deriving stability and decay rates
in the max norm for the scheme of \cite{OlikerPrussner88}.

Other methods do exist such as the vanishing moment method of
X. Feng and M. Neilan \cite{FengNeilan09JSC, FengNeilan09SIAM,
  FengNeilan11}, the numerical moment method of X. Feng et
al. \cite{FengKaoLewis13}, the penalty method of S. Brenner et
al. \cite{BrennerGudiNeilanSung11, BrennerNeilan12}, nonconforming
elements and quadratic elements by M. Neilan \cite{Neilan10,
  Neilan13}, standard finite element methods by G. Awanou \cite{Awanou15},
  least squares and augmented Lagrangian methods of E. Dean
and R. Glowinski \cite{DeanGlowinski03, DeanGlowinski04,
  DeanGlowinski06a, DeanGlowinski06b, SorensenGlowinski10}, and the
$C^1$ finite element method of K. B{\"o}hmer \cite{Bohmer08}.
Error estimates in $H^1(\Omega)$
are established in \cite{BrennerGudiNeilanSung11,BrennerNeilan12}
for solutions $u$ with regularity $H^3$ and above.

The Oliker-Prussner method reads as follows \cite{OlikerPrussner88}. Let $\Nh$ be a 
set of nodes with quasi-uniform spacing $h$ and let $\mathcal{T}_h$
be a subordinate mesh which induces a computational domain
$\dm_h \subset\dm$. Let $\Nhi\subset\dm$ denote the interior nodes and
$\Nhb\subset\partial\Omega$ the boundary nodes. The discrete solution $u_h$ is a 
{\it convex nodal function} defined on $\Nh$ satisfying
$u_h(x_i)=g(x_i)$ for all $x_i\in\Nhb$ and
\begin{equation}\label{discrete}
  |\partial u_h(x_i)| = f_i
  \qquad\forall \, x_i\in\Nhi,
\end{equation}

where $|\cdot|$ denotes the
$d$-dimensional Lebesgue measure, $\partial u_h(x_i)$ is the
subdifferential of $u_h$ at $x_i$ (see \eqref{subdifferential_node}),
and $f_i = \int_{\dm} f \phi_i$ is a suitable averaging of $f$ against
the canonical hat basis functions $\{\phi_i\}$ over $\Th$
\cite{OlikerPrussner88}.
Our primary goal in this paper is to establish a bound for the error
$u-u_h$ in the $L^\infty(\Omega_h)$-norm, which seems missing in the
current literature for \eqref{discrete}.

This endeavor entails developing suitable notions of stability,
consistency, and monotonicity within the max-norm framework.
We now outline the main ingredients. 
We first need to control the $L^{\infty}$-norm of a function $v$
by the size of its subdifferential. This is the celebrated {\it Alexandroff
estimate} \cite{Gutierrez01}, whose discrete counterpart for a
nodal function $v_h$ is established in \cite{NochettoZhang}:
if $\mathcal{C}_h^-(v_h)$ is the (lower) {\it contact set} of nodes,
namely $x_i\in\Nh$ so that $v_h(x_i) = \Gamma (v_h)(x_i)$ with
$\Gamma(v_h)$ the (lower) convex envelope of $v_h$, then
\begin{equation}\label{Alexandroff}
\max_{x_i\in\Nh} v_h^-(x_i) \leq C \left (
\sum_{x_i \in \mathcal{C}_h^-(v_h)} \abs{\partial v_h(x_i)} \right)^{1/d} ,
\end{equation}
where $v_h\ge0$ on $\partial\Omega$ and
$v^-_h(x) = \max\{-v_h(x), 0\}$ is the negative part of $v_h$.
Note that the constant $C=C(d,\Omega)$ is proportional to the diameter
of $\dm$
and the nodal contact set is just a collection of nodes \cite{NochettoZhang}.
We also refer to \cite{KuoTrudinger90, KuoTrudinger92, KuoTrudinger96, KuoTrudinger00} for discrete Alexandroff estimates similar to \eqref{Alexandroff}
and corresponding Alexandroff-Bakelman-Pucci maximum principles for general
fully nonlinear elliptic problems.

Our concept of stability in the max-norm and second ingredient
is the following {\it discrete continuous dependence} estimate,
which is a refinement of \eqref{Alexandroff} and is derived in Section
\ref{S:stability}.
If $v_h$ and $w_h$ are two nodal functions over
$\Nh$ and $v_h(x_i) \ge w_h(x_i)$ for all $x_i \in \Nhb$, then
\[
\max_{x_i\in\Nh} (v_h - w_h)^-(x_i) \leq C \left( \sum_{x_i \in \mathcal{C}_h^-(v_h - w_h)} \Big( \abs{ \partial v_h (x_i)}^{1/d} - \abs{ \partial w_h (x_i)}^{1/d} \Big )^d  \right)^{1/d}.
\]
where $C = C(d, \dm)$.
The proof of such an estimate relies on a combinination of two novel
tools from analysis and geometry, namely the discrete Alexandroff
estimate \eqref{Alexandroff} and the Brunn-Minkowski inequality 
(see Lemma \ref{BM}).
The above estimate will be instrumental to compare the discrete
solution $u_h$ with the nodal function $N_hu$ associated with the
exact solution $u$, which is defined as $N_hu(x_i)=u(x_i)$ for all $x_i\in\Nh$.

The third ingredient is {\it operator consistency}, which is a careful
study of the discrepancy between $u$ and $N_hu$ carried out in Section
\ref{S:consistency}.
We first show that the discrete operator is exact for convex quadratic
polynomials $p$ satisfying $0 < \lambda I \leq D^2p \leq \Lambda I$, namely
\[
  |\partial N_h p(x_i)| = \int_{\dm} \phi_i(x) \det D^2 p(x) dx 
\]
at interior nodes $x_i\in\Nhi$ so that dist $(x_i,\partial\Omega_h) \ge Rh$
with constant $R=R(\lambda,\Lambda)$, provided $\Nhi$
is {\it translation invariant}.
For $u \in C^{2, \alpha}(\overline{\dm})$, we immediately deduce that the consistency error is of order $O(h^{\alpha})$ for $0 < \alpha \leq 1$
\[
\Big| \abs{\partial N_h u(x_i)} - \int_{\dm} \phi_i(x) \det D^2 u(x)
dx \Big| \le C h^{\alpha} |u|_{C^{2,\alpha}(\overline{B_i})}
\int_\Omega \phi_i(x)  dx,
\]
where $B_i := B_{Rh}(x_i)$ is a ball centered at $x_i$ with radius $Rh$.
We can also measure the consistency error in Sobolev norms. If
$u\in W^s_q(\Omega)$ with $ \frac d q + 2 < s \leq 3$, then we exploit
the Sobolev embedding $W^s_q(\Omega) \subset C^{2, \alpha}(\overline{\dm})$
with $\alpha = s - 2 - d/q$ and thus replace
$|u|_{C^{2,\alpha}(\overline{B_i})}$ by $|u|_{W^s_q(B_i)}$.
Since the set of nodes $\Nh$ conforms to $\partial\Omega$, its translation
invariance structure breaks down for nodes close to $\partial\Omega$,
and so do the consistency bounds which become of order $O(1)$.
These nodes are handled differently via a discrete barrier argument
discussed in Section \ref{S:barrier}.

Combining the consistency and stability estimates, together with the
non-degeneracy assumption $f\ge \lambda_F>0$, we derive the following
pointwise convergence rate for $C^{2,\alpha}$-solutions in Section \ref{S:Rates}
\[
\inftynorm{u - \Gamma(u_h)}  \leq C h^{\alpha},
\]
where the constant $C = C(d, \Omega, \lambda, \Lambda, \lambda_F)
\big(\norm{u}_{C^{2, \alpha}(\overline{\dm})} + |u|_{W^2_\infty(\Omega)} \big)$. 
We point out that the $C^{2,\alpha}$-regularity assumption
of $u$ is a consequence of suitable assumptions on $f$.
In fact, if $0 < \lambda_F \le f \le \Lambda_F$, then
$0 < \lambda \le D^2 u \le \Lambda$ for some constant
$\lambda, \Lambda$, and if $f \in C^{\alpha}(\overline{\dm})$ and $\Omega$ is of class
$C^{2, \alpha}$, then $u \in C^{2, \alpha}(\overline{\dm})$
 \cite{Caffarelli90a}, \cite[Section 4.3]{Gutierrez01}.
 We also stress that the discrete barrier argument for 
nodes close to $\partial\Omega$ is
responsible for the semi-norm $|u|_{W^2_\infty(\Omega)}$ in the error estimate.

In addition, we prove in Section \ref{S:Rates}
the following pointwise error estimate for functions with
$W^s_q$-regularity and $ \frac d q + 2 < s \leq 3$
\[
  \inftynorm{u - \Gamma(u_h)} \leq C h^{s-2} ,
\]
where the constant $C = C(d, \dm, \lambda, \Lambda,\lambda_F) \big(
|u|_{W^s_q(\dm)} + |u|_{W^2_\infty(\Omega)} \big)$. This estimate
shows that the discrete method \eqref{discrete} exhibits first order accuracy in the
max-norm provided $u\in W^3_q(\Omega)$ with $q>d$. Since
$u\in W^3_q(\Omega)\subset C^2(\overline{\dm})$, the above error
estimate is still in the realm of classical solutions. However, our
results extend to solutions
$u \in W^{2,\infty}(\overline{\dm})\setminus C^2(\overline{\dm})$ whose
Hessian $D^2 u$ is discontinuous across a set $S$ of (box) dimension $n < d$:
\[
  \inftynorm{u - \Gamma(u_h)} \leq C h^{s-2} \, |u|_{W^s_q(\Omega \setminus S)}
  +C h^{\frac{d-n}{d}} \, |u|_{W^2_\infty(\Omega)}.
\]

The rest of this paper is organized as follows.
In Section \ref{S:discreteconvexity} we discuss our notion of discrete
convexity, a topic that has received considerable attention recently,
and explore properties of subdifferentials.
In Section \ref{S:soln_ap}, we introduce the Alexandroff and
viscosity solution concepts for the \MA equation \eqref{pde}
as well as the geometric method \eqref{discrete}.
We prove stability of \eqref{discrete} in Section \ref{S:stability}
and consistency in Section \ref{S:consistency}. We conclude with
three rates of convergence depending on solution regularity in Section \ref{S:Rates}.

%
	\section{Discrete Convexity}\label{S:discreteconvexity}
%

Approximating convex solutions of the \MA equation \eqref{pde} entails two
  essential difficulties: dealing with discrete convexity and the
  fully nonlinear nature of \eqref{pde}. The former issue
 has been
investigated in
\cite{ChoneLeMeur01, CarlierLachandRobertMaury01, AguileraMorin09,
   Oberman13, MerigotOudet14}, 
and \cite{ChoneLeMeur01} shows that convex piecewise
linear functions over a sequence of shape-regular meshes obtained by
uniform refinement of a fixed mesh may not be dense in the
set of convex functions. In fact, the Lagrange interpolant of a
convex function may not even be convex. Several notions of discrete
convexity have been proposed in the literature:
\cite{CarlierLachandRobertMaury01} deals with convex function
interpolation on given meshes;
\cite{AguileraMorin09} introduces finite element functions with
positive weak Hessian;
\cite{Oberman13} imposes positive second order finite differences in
all directions within a given stencil. In this paper, we deal with
nodal functions and say they are convex if they admit a supporting
hyperplane at every node. We make this explicit below.

\subsection{Nodes and Meshes}

In contrast to mesh-based methods, the discretization of
\eqref{pde} hinges on a collection of nodes
$\mathcal{N}_h := \Nhi\cup\mathcal{N}_h^{\partial}$ so that
\[
\Nhi : = \{x_i\}_{i=1}^n \subset\Omega,
\qquad
\mathcal{N}_h^{\partial} := \{x_i\}_{i=n+1}^N \subset\partial\Omega,
\]
and a collection of simplices (or elements) $T$ with nodes $x_i$
which form a conforming mesh $\Th$ of $\Omega$ and determine the computational
domain $\Omega_h := \cup_{T\in\Th} T$; since $\Omega$ is convex we infer
that $\Omega_h\subset\Omega$.
For each element $T$, we denote by $h_T$ the diameter of $ T$ and
by $\rho_T$ the diameter of the largest inscribed ball in $T$.
For each node $x_i$, we define the local spacing at $x_i$ as
\[
h_i := \max\{ h_T: \, x_i \in T \}.
\]
We say that the nodal set $\Nh$ is {\it quasi-uniform} if there exist constants
$0< \gamma \leq 1$ and $h > 0$ such that $\gamma h \leq h_i \leq h$
for all $1 \le i \le n$.
We define the shape-regularity constant of an element $T$ to be
$
  \sigma_T := \frac{h_T}{\rho_T} 
$
and we say that $\Nh$ is {\it shape-regular}
if there exists $\sigma>0$ such that $\sigma_T \leq \sigma$ for all elements $T$.
We will assume
throughout that $\Nh$ is quasi-uniform and shape-regular.

We say that $\Nhi$ is {\em translation invariant} if there is a
basis $\{e_j\}_{j=1}^d$ of $\mathbb R^d$ with $|e_j| \leq 1$
for all $1\leq j \leq d$ so that
\begin{equation}\label{translation-invariant}  
  \Nhi = \left\{ x_i = h \sum_{j=1}^d k_j e_j: k_j \in \mathbb Z \right\} \cap \dm.
\end{equation}
For the boundary nodes $\Nhb$ we only require that
\[
  \bdry \subset \cup_{x_i \in \Nhb} B_{h/2}(x_i).
\]
Obviously, a cartesian lattice with spacing $\sqrt{d}h$ is translation
invariant and $\{e_j\}_{j=1}^d$ are the canonical unit vectors in $\mathbb{R}^d$. 

We denote by $\{\phi_i\}_{i=1}^n$
the canonical basis of piecewise linear functions associated with
$\Nhi$ over $\Th$. We say that $\{\phi_i\}_{i=1}^n$ is {\it translation
invariant} provided that for all $x_i,x_j\in\Nhi$ such that
dist$(x_i,\bdry_h)$, dist$(x_j,\bdry_h)>h$ and for all $x\in\mathbb{R}^d$ we have
\begin{align}\label{meshinvariant}
\phi_i(x+x_i) = \phi_j(x + x_j).
\end{align}
If $\Nhi$ is translation invariant, then so is $\{\phi_i\}_{i=1}^n$ for
a suitable mesh $\Th$. Since the construction of such $\Th$ is obvious for
$d=2$, we now examine the case $d=3$.
Take a node $z \in \Nhi$ with dist$(z, \bdry_h) > \sqrt{d}h$ and
consider the box
\[
Q = \Big\{x \in \dm: \, x = z + \sum_{j=1}^d t_j e_j, \; 0 \leq t_j \leq 1 \Big\}.
\]
We would like to divide that box into a set of six disjoint
tetrahedra $\{T_i \}_{i=1}^6$ such that $\cup_{i=1}^6 T_i = Q$. 
To do so, we label the eight nodes of $Q$ as follows:
\begin{gather*}
v_0 = z, \quad v_1 = z + e_3, \quad v_2 = z + e_1 + e_3,
\quad v_4 = z + e_2, \quad v_3 = z + e_1,
\\
v_5 = z+ e_2 + e_3, \quad v_6 = z + e_1 + e_2 + e_3,
\quad v_7 = z+e_1+e_2.
\end{gather*}
Let $\{T_i \}_{i=1}^6$ be the convex hulls of the given nodes
\begin{align*}
T_1 = \hull \{v_0, v_1, v_3, v_4\},
~
T_2 = \hull \{v_1, v_2, v_4, v_5\},
~
T_3 = \hull \{v_1, v_2, v_3, v_4\},
\\
T_4 = \hull \{v_6, v_2, v_5, v_7 \},
~
T_5 = \hull \{v_7, v_4, v_3, v_2 \},
~
T_6 = \hull \{v_7, v_4, v_5, v_2 \}.
\end{align*}
We realize that opposite faces are cut into two compatible triangles, e.g. faces
$\hull \{v_0,v_1,v_2,v_3\}$ and $\hull \{v_4,v_5,v_6,v_7 \}$ are cut
along the segments $\hull \{v_1,v_3\}$ and $\hull \{v_5,v_7\}$.
Finally, for every node $x_i\in\Nhi$, we build $Q_i = Q - z + x_i$ and corresponding
six tetrahedra, and observe that this construction yields a conforming mesh
$\Th$ satisfying \eqref{meshinvariant}.

\subsection{Convexity of Nodal Functions}\label{S:nodal-functions}
We say that the nodal function $u_h:\mathcal {N}_h \to \mathbb{R}$ is 
{\em convex} if for all $x_i \in \Nhi$ 
there is a supporting hyperplane $L$ of $u_h$, that is
\[
  L(x_j) \leq u_h(x_j) \text{ for all $x_j \in \mathcal{N}_h$ and } L(x_i) = u_h(x_i).
\]
We define the convex envelope of a nodal function $u_h$ to be
\begin{align}
  \label{convexenvelope}
  \Gamma (u_h)(x) = \sup_{L \text{ affine}} \{ L(x): \; L(x_i) \leq u_h(x_i)
  \quad \forall x_i \in \Nh \} \quad \forall x \in \dm_h.
\end{align}
If the nodal function $u_h$ is convex, then we have
\begin{align}\label{extension}
  u_h(x_i) = \Gamma (u_h)(x_i) \quad \text{for all $x_i \in \Nhi$.}  
\end{align}
We regard $\Gamma (u_h)$ as a natural convex extension of
$u_h$ and call it the {\it convex interpolant} of $u_h$.
On the other hand, given a continuous function
$u:\overline{\Omega}\to\mathbb{R}$ we denote by
$N_h u:\Nh\to\mathbb{R}$ the nodal function associated with $u$:
\begin{equation}\label{nodal-function}
  N_h u(x_i) = u(x_i)
  \quad\forall \, x_i\in\Nh.
\end{equation}

Since $\Gamma (u_h)$ is a piecewise linear function, it induces a
mesh $\Th$ which depends on $u_h$ and is in general different from the
original mesh satisfying \eqref{meshinvariant}; see Figure \ref{fig:anisotropy}.
Given a convex nodal function $u_h$
and a node $x_i \in \Nhi$,
we define the set $A_i(u_h)$ of {\it adjacent nodes}
(or {\it adjacent set}) of $x_i$ for $u_h$ as
a collection of nodes $x_j \in \Nh$ closest to $x_i$ such that there exists a
supporting hyperplane $L$ of $u_h$ at $x_i$ and $L(x_j) = u_h(x_j)$.
The set $A_i(u_h)$ is the collection of nodes in the star associated
with $x_i$ in the mesh $\Th$ induced by $\Gamma(u_h)$.
The following example illustrates that such a star could be quite
elongated even for a cartesian lattice $\Nhi$, especially if
the Hessian $D^2u$ is degenerate or nearly degenerate.
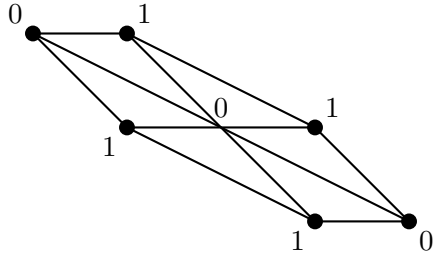
\begin{figure}[h!]
\begin{center}
\begin{tikzpicture}[scale = 1.25]
\draw[-,thick](1,0)--(-1,1)--(-2,1)--(-1,0)--(1,-1)--(2,-1)--(1,0);
\draw[-,thick](1,0)--(0,0)--(-1, 0);
\draw[-,thick](-1,1)--(0,0)--(1, -1);
\draw[-,thick](-2,1)--(0,0)--(2, -1);
\node[inner sep = 0pt,minimum size=6pt,fill=black!100,circle] (n2) at (1,0)  {};
\node[inner sep = 0pt,minimum size=6pt,fill=black!100,circle] (n2) at (-1,1)  {};
\node[inner sep = 0pt,minimum size=6pt,fill=black!100,circle] (n2) at (-2,1)  {};
\node[inner sep = 0pt,minimum size=6pt,fill=black!100,circle] (n2) at (-1,0)  {};
\node[inner sep = 0pt,minimum size=6pt,fill=black!100,circle] (n2) at (1,-1)  {};
\node[inner sep = 0pt,minimum size=6pt,fill=black!100,circle] (n2) at (2,-1)  {};
\node at (1,0) [above right] {1};
\node at (-1,0) [below left] {1};
\node at (-1,1) [above right] {1};
\node at (1,-1) [below left] {1};
\node at (0,0) [above] {0};
\node at (-2,1) [above left] {0};
\node at (2,-1) [below right] {0};
\end{tikzpicture}
\caption{Anisotropic star at $(0,0)$ induced by the convex envelope $\Gamma(u_h)$
of the nodal function $u_h(x_i) = (x_i \cdot e)^2$ where $e = (1,2)$.
Note that $\Gamma (u_h)(x) = |x \cdot e|$ in the star.}
\label{fig:anisotropy}
\end{center}
\end{figure}
\begin{example}[anisotropic star]\label{E:anisotropic-star}
\rm
Let $\dm= \mathbb{R}^2$, $\mathcal{N}_h = \mathbb{Z}^2$ and $h=1$. Let 
$u(x) = (x \cdot e)^2$ where $e = (1,m)$ for some integer $m\ge1$ and
$u_h(x_i) = u(x_i)$ for all $x_i \in \mathcal{N}_h$. 
Then the convex envelope $\Gamma(u_h)$ of $u_h$ induces an anisotropic
mesh $\Th$; Figure \ref{fig:anisotropy} displays the star associated
to the origin for $m=2$.
The convex envelope $\Gamma(u_h)$ in such a star is $\Gamma(u_h)(x)=|x \cdot e|$. 
\end{example}

Given a mesh $\Th$ with nodes $\Nh$
we denote by $I_h(u_h)$ the Lagrange interpolant $I_h(u_h)$ of $u_h$ over
$\Th$, namely the continuous piecewise linear
function that interpolates the nodal values of $u_h$ over $\Th$.
The following property is helpful to check discrete convexity:
given a mesh $\Th$ with nodes $\Nh$ and a nodal function 
$u_h$, $I_h(u_h)$
is convex if and only if it satisfies \cite[Lemma 5.3]{NochettoZhang}
\begin{align}\label{sign} 
  \jump{\gradv I_h(u_h)}|_F \geq 0 \text{ for all faces $F$}
\end{align}
where $F = T^+\cap T^-$ with $T^\pm\in\Th$, the jump is given by
\begin{equation}\label{jump-def}
\jump{\gradv I_h(u_h)}|_F
:= - n_F^+ \cdot \gradv  I_h(u_h)|_{T_+} - n_F^- \cdot \gradv  I_h(u_h)|_{T_-}
\end{equation}
and $n_F^{\pm}$ are the outer normal vectors of $T_{\pm}$ on face $F$.
If $I_h(u_h)$ is convex, then $I_h(u_h)=\Gamma(u_h)$.

We finally let the (nodal lower) {\it contact set} of a nodal
function $u_h$ be
\begin{align}
  \label{contactset}
  \mathcal{C}^-_h(u_h) := \{ x_i \in \Nhi : \; \Gamma (u_h)(x_i) = u_h(x_i) \}.
\end{align}
Note that if $u_h$ is convex, then $\mathcal{C}^-_h(u_h) = \Nhi$; otherwise,
$\mathcal{C}^-_h(u_h) \subset \Nhi$.

\subsection{Subdifferential}

Let $u : \Omega \rightarrow \mathbb{R}$ be a convex function and $x_0 \in \Omega$. 
The subdifferential of $u$ at $x_0$ is the set
\begin{align}\label{subdifferential_function}
  \partial u(x_0) = \{ v \in \mathbb{R}^d:  u(x) \ge u(x_0) + v \cdot (x - x_0) \}.
\end{align}
Given a set $S \subset \Omega$, we define 
\[
  \partial u(S) = \cup_{x\in S} \partial u(x).
\]
Since $u$ is a convex function, the subdifferential $\partial u(x)$
is non-empty and convex for all $x \in \Omega$.

Similarly, we define the subdifferential of nodal function $u_h$ at node $x_i$ by  
\begin{align}
  \label{subdifferential_node}
  \partial u_h(x_i) = \{v \in \mathbb{R}^d:\  u_h(x_j) \ge u_h(x_i)+ v \cdot (x_j-x_i)\ \quad \forall x_j \in \mathcal{N}_h \}.
\end{align}
If the nodal function $u_h$ is convex, then $\partial u_h(x_i)$ is
non-empty for all $x_i \in \Nhi$. The following lemma
relates  definitions \eqref{subdifferential_function} and \eqref{subdifferential_node}.

\begin{lemma}[discrete subdifferential]\label{L:disc-subd}
If $u_h$ is a convex nodal function, then $\partial u_h(x_i) = \partial \Gamma(u_h)(x_i)$ for all $x_i \in \Nhi$. 
\end{lemma}
\begin{proof}
  Obviously, if $v \in  \partial \Gamma(u_h)(x_i)$, that is, 
\[
  \Gamma(u_h)(x) \geq  \Gamma(u_h)(x_i) + v\cdot (x - x_i)  \quad \forall x \in \dm,
\]
then $v \in \partial u_h(x_i)$ thanks to \eqref{extension}.

Conversely, let $\Th$ be a mesh induced by the convex envelope
$\Gamma(u_h)$, let $T \in \Th$ be an element with vertices $\{x_j\}$.
If $v \in \partial u_h(x_i)$, then
$
   u_h(x_j) \ge u_h(x_i)+ v \cdot (x_j-x_i)
$
for all $x_j \in \Nh$, whence again thanks to \eqref{extension}, we have
$
   \Gamma(u_h)(x_j) \ge \Gamma(u_h)(x_i)+ v \cdot (x_j-x_i)
$
for all vertices $x_j$ of $T$.
Since $\Gamma(u_h)$ is linear in $T$, we have
$
   \Gamma(u_h)(x) \ge \Gamma(u_h)(x_i)+ v \cdot (x-x_i)
$
for any $x \in T$. This shows that $v \in \partial \Gamma (u_h)(x_i)$ as well. 
\end{proof}

\begin{lemma}[subdifferential monotonicity] \label{monotonicity}
Let $u_h$ and $v_h$ be two convex nodal functions. If $u_h(x_i) \leq v_h(x_i)$ and $u_h(x_j) = v_h(x_j)$ for all $j \neq i$, then
\begin{align*}
\partial v_h(x_i) \subset \partial u_h(x_i)
\quad \text{and} \quad
\partial u_h(x_j) \subset \partial v_h(x_j)
\text{ for all $j \neq i$.}
\end{align*}
\end{lemma}
\begin{proof} This follows directly from the definition of
subdifferential \eqref{subdifferential_node}.
\end{proof}  

Given two compact sets $A,B$ of $\mathbb{R}^d$, their {\it Minkowski sum} is
given by
\begin{equation}\label{minkowski-sum}
A + B := \{ v + w \in \mathbb{R}^d:  v \in A \text{ and } w \in B \}.
\end{equation}
\begin{lemma}[addition of subdifferentials]\label{Msum}
If $u_h$ and $v_h$ be two convex nodal functions, then 
    \begin{align*}
\partial w_h(x_i) + \partial v_h(x_i) \subset \partial (w_h + v_h)(x_i) 
\text{ for all $x_i \in \Nhi$}.
\end{align*}
\end{lemma}

\begin{proof}
We note that if $w \in \partial w_h(x_i)$ and $v \in \partial v_h(x_i)$, then
\[
  w_h(x_j) \ge w_h(x_i)+ w \cdot (x_j-x_i)
 \quad \text{and} \quad
  v_h(x_j) \ge v_h(x_i)+ v \cdot (x_j-x_i) 
\]
for all $x_j \in \Nh$. Adding both inequalities yields
\[
 w_h(x_j) + v_h(x_j) \ge w_h(x_i) + v_h(x_i) + (w+v) \cdot (x_j-x_i)
\]
which implies $w + v \in \partial (w_h + v_h)$.
\end{proof}

Computing the subdifferential $\partial u_h(x_i)$ of a
given convex nodal function $u_h$ at $x_i\in\Nhi$
is a nontrivial task because it is nonlocal. In fact, Lemma
\ref{L:disc-subd} shows that it involves computing the convex envelope
$\Gamma(u_h)$ of $u_h$. The following lemma, proved in \cite[Lemma 5.4]{NochettoZhang},
characterizes $\partial\Gamma(u_h)$.
\begin{lemma}[characterization of subdifferential]\label{char_subdifferential}
Let $u_h$ be a convex nodal function and $\Th$ be the mesh
induced by its convex envelope $\Gamma(u_h)$.
Then the subdifferential of $u_h$ at $x_i\in\Nhi$ is the convex hull
of the constant gradients
$
\gradv \Gamma(u_h)|_T
$
for all $T \in \Th$ which contain $x_i$.
\end{lemma}
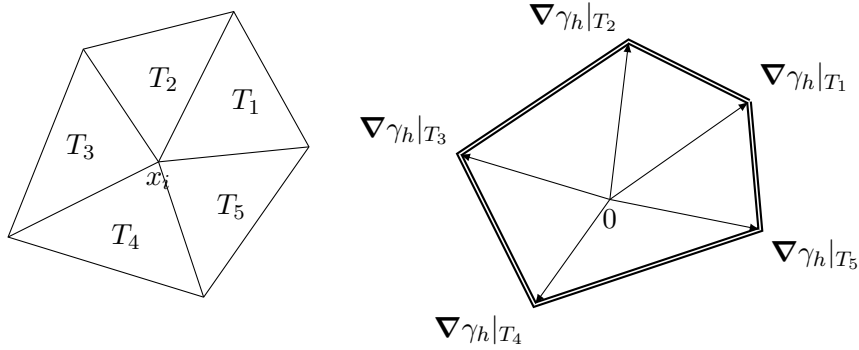
\begin{figure}[h!]
\begin{tikzpicture}
  \coordinate [label=below:$x_i$] (z1) at (1,1);
  \coordinate  (z2) at (3,1.2);
  \coordinate  (z3) at (2,3);
  \coordinate  (z4) at (0,2.5);
  \coordinate  (z5) at (-1,0);
  \coordinate  (z6) at (1.6,-0.8);
  \coordinate [label=below left:$T_1$] (k1) at (2.5, 2.1);
  \coordinate [label=below left:$T_2$] (k3) at (1.4,2.4);
  \coordinate [label=below left:$T_3$] (k4) at (0.3,1.5);
  \coordinate [label=below left:$T_4$] (k5) at (0.9,0.3);
  \coordinate [label=below left:$T_5$] (k2) at (2.3, 0.7);
  \coordinate [label=below :$0$] (o) at (7,0.5);
  \coordinate [label=above right:$\gradv \gamma_h|_{T_1}$] (g1) at (8.85,1.8);
  \coordinate [label=above left :$\gradv \gamma_h|_{T_2}$] (g2) at (7.25,2.6);
  \coordinate [label=above left:$\gradv \gamma_h|_{T_3}$] (g3) at  (5,1.1);
  \coordinate [label=below left:$\gradv \gamma_h|_{T_4}$] (g4) at  (6,-0.9);
  \coordinate [label=below right:$\gradv \gamma_h|_{T_5}$] (g5) at  (9,0.1);  
  \draw  (z1) -- (z2)  ;
  \draw  (z1) -- (z3)  ;
  \draw  (z1) -- (z4)  ;
  \draw  (z1) -- (z5)  ;
  \draw  (z1) -- (z6)  ;
  \draw  (z6)- - (z5) -- (z4) -- (z3) -- (z2) -- (z6);
  \draw [double, thick](g1) -- (g2) -- (g3) -- (g4) -- (g5) -- (g1);
  \draw [-latex] (o) -- (g1);
  \draw [-latex] (o) -- (g2);
  \draw [-latex] (o) -- (g3);
  \draw [-latex] (o) -- (g4);
  \draw [-latex] (o) -- (g5);
\end{tikzpicture}
\caption{Star centered at node $x_i$ corresponding to the mesh $\Th$
induced by the convex envelope $\gamma_h=\Gamma(u_h)$ of $u_h$
and subdifferential $\partial u_h(x_i)$ of the convex
nodal function $u_h$ at node $x_i$ such that $0\in\partial u_h(x_i)$.
The latter is the convex hull of the constant
element gradients $\nabla \gamma_h|_{T_j}$ for $1\le j\le 5$.}
\label{F:subdifferential}
\end{figure}

Figure \ref{F:subdifferential} depicts the subdifferential
$\partial u_h(x_i)$ of the
convex nodal function $u_h$ at node $x_i$ for $d=2$.
Computing $\partial u_h(x_i)$ is equivalent to finding
a mesh $\Th$ such that the jumps $\jump{\gradv I_h(u_h)}|_F$
of the Lagrange interpolant of $I_h(u_h)$ on faces $F$ are nonnegative. 

Unfortunately, the notion of discrete convexity is not sufficiently
robust geometrically. Consider a pair of elements $T^\pm\in\Th$ with
common face $F=T^+ \cap T^-$, and assume $\jump{\gradv I_h(u_h)}|_F=0$.
Increasing the values of $u_h$ at the nodes opposite to $F$ increases
the jump, and thus preserves convexity of $I_h(u_h)$ on $\Th$, whereas
decreasing the values of $u_h$ violates convexity on $\Th$.
\begin{lemma}[geometric stability]\label{L:geo-stab}
Let $u_h$ be a nodal function. If the convex envelope $\Gamma(u_h)$
of $u_h$ has positive jumps
on all faces of the induced mesh $\Th$, then $\Th$ is invariant under
small nodal perturbations of $u_h$.
\end{lemma}
\begin{proof}
Since the jumps $\jump{\gradv \Gamma(u_h)}|_F$ are continuous with
respect to nodal variations of $u_h$, and they are positive, they
remain positive under small nodal perturbations. We then apply
\eqref{sign} to deduce the assertion.
\end{proof}

%
\section{Solution and Approximation}\label{S:soln_ap}
%

There are two notions of weak solutions of \eqref{pde}: the Alexandroff
solution hinges on a geometric interpretation of \eqref{pde} and the
viscosity solution relies on the comparison principle. We review these
definitions now and discuss a geometric approximation of \eqref{pde}
due to Oliker and Prussner \cite{OlikerPrussner88}.
  
%
	\subsection{Alexandroff Solution}
%

To motivate this solution concept, suppose for the moment that
$u\in C^2(\dm)$ is a strictly convex function satisfying
\eqref{pde} in a classical sense. The convexity and $C^2$-regularity
assumptions imply that $\partial u(x) = \gradv u(x)$
and that the subdifferential, viewed as a map
$\partial u:\Omega\to\mathbb R^d$, is injective.
Consequently, the change of variables $y=\nabla u(x)$ reveals that
\begin{align*}
\int_D f\, dx = \int_D \det D^2 u(x)\, dx =
\int_{\partial u(D)} \, dy = |\partial u(D)|
\end{align*}
for all Borel sets $D\subset \Omega$, 
where $\partial u(D) = \cup_{x\in D} \partial u(x)$ and $|\cdot|$ is
the $d$-dimensional Lebesgue measure. Since $\partial u$ is
well-defined for non-smooth convex functions, the above identity
allows one to widen the class of admissible solutions
\cite[Section 1.2]{Gutierrez01}.

\begin{definition}[\MA measure]
We define the \MA measure associated with a convex function $u \in C(\Omega)$ as
\[
Mu(D) = |\partial u (D)|
\]
for any Borel set $D$,
where $|\cdot |$ denotes the $d$-dimensional Lebesgue measure.
\end{definition}

\begin{definition}[Alexandroff solution] \label{D:Alexandroff}
Let $\mu$ be a Borel measure defined in $\Omega$, an open and convex
subset of $\mathbb{R}^d$.
We say a convex function $u \in C(\Omega)$ is an Alexandroff solution to the \MA equation 
\[
\det D^2 u = \mu
\]
if the \MA measure $Mu$ associated with $u$ equals $\mu$.
\end{definition}

The Alexandroff solution is closed under uniform convergence.
This is stated in the lemma below and its proof is given in
\cite[Lemma 1.2.3]{Gutierrez01}.
\begin{lemma}[weak convergence of \MA measures]\label{weakconvergence}
If $u_k$ are convex functions in $\dm$ such that $u_k \rightarrow u$
as $k\to\infty$ uniformly on compact subsets of $\dm$, then the
associated \MA measures $Mu_k$ tend to $Mu$ weakly, that is
\[
\int_{\dm} \phi(x) d Mu_k(x) \rightarrow \int_{\dm} \phi(x) d Mu(x)
\quad\text{as } k\to\infty,
\]
for every $\phi$ continuous with compact support in $\dm$. 
\end{lemma}

%
\subsection{Viscosity Solution}
%
This solution concept hinges on the comparison principle.\looseness=-1

\begin{definition}[viscosity solution]\label{D:viscositysolution}
Let $u \in C(\Omega)$ be a convex function and
$f \in C(\Omega), f \ge 0$.
The function $u$ is a viscosity sub-solution (super-solution) of
the \MA equation in $\Omega$ if whenever a convex function $\phi \in C^2(\Omega)$
and $x_0 \in \Omega$ are such that $(u - \phi)(x) \le (\ge) (u -
\phi)(x_0)$ for all $x$ in a neighborhood of $x_0$, then we must have
\[
\det D^2 \phi (x_0) \ge (\le) f(x_0).
\]
\end{definition}

If $u\in C(\Omega)$ is an
Alexandroff solution with $Mu=f$ and $f\in C(\Omega)$, then $u$ is a
viscosity solution \cite[Proposition 1.3.4]{Gutierrez01}.
Conversely, if $f \in C(\overline{\dm})$ and $f>0$, then the viscosity
solution is also the Alexandroff solution \cite[Proposition 1.7.1]{Gutierrez01}.

\subsection{Examples of Weak Solutions}
We show examples of Alexandroff and viscosity solutions, which are not
classical solutions, namely $u\notin C^2(\overline{\Omega})$.

\begin{example}[Alexandroff solution]\label{E:Alexandroff}
\rm
Let $\dm = B_1 (0) \subset \mathbb R^2$ and
\[
u (x) = |x| - 1 .
\]
The function $u(x)$ is an Alexandroff solution of the Monge-Amp\`{e}re equation 
\[
\det D^2 u = \pi \delta_{(x=0)},
\]
where $\delta_{(x=0)}$ is the Dirac measure at the origin; $u$ is not a viscosity solution.
\end{example}

\begin{example}[viscosity solution]\label{E:viscosity}
\rm
Let $\dm = B_2 (0) \subset \mathbb R^2$ and \cite{FroeseOberman11SIAM}
   \begin{align}\label{example2}
 u(x)  = \left\{
\begin{array}{ll}
\frac 12 |x|^2  \quad &\text{in $|x| \leq 1$},  
\\
\frac 12 |x|^2 + \frac 12 \big( |x| -1 \big)^2
\quad &\text{in $ 1 \leq |x| \leq 2$}.
\end{array} \right.
\end{align}
The function $u$ is a viscosity solution of the Monge-Amp\`{e}re equation
\begin{align*}
 \det D^2 u(x)  = \left\{
\begin{array}{ll}
1  \quad &\text{in $|x| \leq 1$},  
\\
4 - 2|x|^{-1} \quad &\text{in $ 1 \leq |x| \leq 2$}.
\end{array} \right.
\end{align*}
We note that
$u \in C^{1,1}(\overline{\Omega}) \setminus C^2(\overline{\dm})$
and the Hessian of $D^2u$ exhibits a jump discontinuity across
$\partial B_1(0)$.
\end{example}

\subsection{Oliker-Prussner Method}
Following \cite{OlikerPrussner88} we can approximate \eqref{pde}
exploiting its geometric interpretation.
Given a quasi-uniform and shape regular nodal set $\Nh$,
corresponding conforming mesh $\Th$ and canonical basis functions
$\{\phi_i\}_{i=1}^n$ associated with $\Nhi$, for any function $f\ge0$ we define the
nodal function $f_h:\Nh\to\mathbb{R}$ to be
\begin{equation}\label{fi}
f_i := \int_{\omega_i} f(x) \phi_i(x) dx
\quad\forall x_i\in\Nhi,
\end{equation}
with $\omega_i=\text{supp }(\phi_i)$ being the star corresponding to $x_i$.

The discretization of \eqref{pde} reads as follows:
we seek a convex nodal function $u_h$ satisfying
\begin{align}\label{FEM}
  \abs { \partial u_h(x_i)} = f_i
  \quad\forall \, x_i \in \Nhi.
\end{align}
We refer to \cite{OlikerPrussner88} for a proof of existence and
uniqueness of \eqref{FEM} for $d=2$. Below we give a proof of existence
for $d\ge2$ that ties the concepts developed so far together. We observe
that definition \eqref{fi} is different from that in \cite{OlikerPrussner88},
which replaces the functions $\phi_i$ by characteristic functions of disjoint
sets containing the nodes $x_i$,
and that $\{\phi_i\}_{i=1}^n$ need not be translation invariant for $u_h$
to exist. This property is instrumental later to derive consistency.

\subsection{Existence: Discrete Perron's Method}
We construct a monotone sequence $\{u_h^k\}_{k=0}^\infty$ of convex nodal functions,
namely
\[
u_h^{k+1}(x_i) \ge u_h^k(x_i) \quad\forall \, x_i\in\Nhi,
\]
which converges to a solution of \eqref{FEM} as $k\to\infty$.
For each $k\ge0$ there is a
mesh $\mathcal{T}_h^k$ with nodes $\Nh$ but possibly different
connectivity than $\mathcal{T}_h^j$ for $j<k$ and the property
\[
I_h^k(u_h^k) = \Gamma(u_h^k).
\]
Therefore, the interpolant $I_h^k(u_h^k)$ of $u_h^k$ over $\mathcal{T}_h^k$
is convex. We illustrate how these meshes $\mathcal{T}_h^k$
change with the iteration counter $k$.

\medskip\noindent
{\bf Construction of $u_h$.} We first initialize the iteration.
We assume that  $\dm$ is contained in a ball $B_R(0)$ of radius $R$
centered at the origin and $d\ge2$. We consider the quadratic polynomial
$p(x) := \frac{\Lambda^{1/d}}{2} \left( |x|^2 - 2R^2 \right)$ with
$\Lambda>0$ to be specified later. Then
\[
\det D^2 p(x) = \Lambda \quad\forall \, x\in \dm,
\qquad 
p(x) \leq g(x) \quad\forall \, x\in\bdry,
\]
provided $-\frac{\Lambda^{1/d}}{2} R^2 \le g(x)$ for all $x\in\bdry$.
We define $u_h^0=N_h p$, namely
\[
u_h^0(x_i) := p(x_i)
\quad\forall \, x_i \in \Nh.
\]
Let $\T_h^0$ be a Delaunay triangulation
associated with $\Nh$, which exists according to \cite[Theorem 2.3]{ChenXu04}.
For such a mesh, the Lagrange
interpolant of $u_h^0$ is convex \cite{ChenHolst11,  ChenXu04, Edelsbrunner00},
and thus coincides with the convex envelope of $u_h^0$:
\[
I_h^0(u_h^0) = \Gamma(u_h^0).
\]
We stress that $\T_h^0$ is in general different from the mesh $\Th$ used to
set \eqref{fi} and \eqref{FEM}.
We assert that $u_h^0$ is a subsolution of \eqref{FEM}. To see this we
need to estimate the measure of the subdifferential $|\partial u_h^0(x_i)|$
at every node $x_i\in\Nhi$. Since $\T_h^0$ is shape-regular, the star
$\omega_i^0$ of $\mathcal{T}_h^0$ around $x_i$ has a diameter
proportional to $h$ and contains a ball
$B_{\alpha h}(x_i)$ of radius $\alpha h$ with $\alpha>0$ only
dependent on shape regularity. Upon subtracting
the affine function $L(x) = p(x_i) + \nabla p(x_i) \cdot (x-x_i)$
from $p$, we can
assume that $p$ and its gradient vanish at $x=x_i$ without changing
$|\partial u_h^0(x_i)|$. Since $p(x) = \frac{\Lambda^{1/d}}{2}|x-x_i|^2$,
we see that $p(x) \ge \frac{\Lambda^{1/d}}{2} \alpha^2 h^2$ on
$\partial\omega_i^0$ and we can apply the discrete Alexandroff estimate
\eqref{Alexandroff} to $u_h^0 - \frac{\Lambda^{1/d}}{2} \alpha^2 h^2$
on $\omega_i^0$ to deduce
\[
|\partial u_h^0(x_i)|^{1/d} \ge C h^{-1} \sup_{\omega_i} \Big( u_h^0 -
\frac{\Lambda^{1/d}}{2} \alpha^2 h^2 \Big)^-
\ge C \Lambda^{1/d} h,
\]
whence
\[
|\partial u_h^0(x_i)| \ge C\Lambda h^d = C\Lambda |\omega_i^0|.
\]
Recalling from \eqref{fi} that $f_i=\int_{\omega_i} f\phi_i \le \Lambda_F |\omega_i|$,
and realizing that $|\omega_i|$ and $|\omega_i^0|$ are comparable,
we can choose $\Lambda>0$ sufficiently large so that
\begin{equation}\label{initial-subd}
|\partial u_h^0(x_i)| \ge f_i
\quad\forall x_i\in\Nhi.
\end{equation}
Moreover, in view of $u_h^0\le I_h^0 g$ on $\partial\Omega_h$ we infer that
$u_h^0$ is a subsolution of \eqref{FEM}, as asserted.

To construct the iterates $u_h^{k+1}$ for $k\ge0$ we recall Lemma \ref{monotonicity}
(subdifferential monotonicity): if the nodal value $u_h^k(x_i)$ increases, then
$|\partial u_h^k (x_i)|$ decreases and $|\partial u_h^k (x_j)|$
increases for all other nodes $x_j$ ($j \neq i$). With this in mind,
we first set the boundary values one at a time
\[
u_h^1(x_i) = g(x_i)
\quad\forall \, x_i \in\mathcal{N}_h^\partial,
\]
which preserves \eqref{initial-subd} and the convexity of $u_h^1$.
If a convex nodal function $u_h^k$ has been computed, we
now describe how to construct $u_h^{k+1}$ upon increasing the internal
nodal values $u_h^k(x_i)$ one at a time. Having determined
$u_h^{k+1} (x_j)$ for $j<i$, we thus define $u_h^{k+1} (x_i)$
to be the largest value so that
\[
|\partial u_h^{k+1} (x_i)| = f_i
\quad\forall \, x_i\in\Nhi,
\]
provided $u_h^{k+1} (x_j) = u_h^k (x_j)$ for $j>i$,
which in turn implies for $j\ne i$
\[
|\partial u_h^{k+1} (x_j)| \ge f_j
\quad\forall \, x_j \in\Nhi.
\]
This is always doable because $|\partial u_h^{k+1}(x_i)|$
decreases continuously with increasing $u_h^{k+1}(x_i)$ until it vanishes,
which corresponds to having a supporting hyperplane at $x_i$ that
touches $u_h^{k+1}$ at $d+1$ nodes distinct from $x_i$ and not lying
in one hyperplane. In addition, since $f_i\ge0$ for all $x_i\in\Nhi$,
the intermediate iterates leading to $u_h^{k+1}$ are always convex by definition.
This process creates a monotone sequence $\{u_h^k\}_{k\in\mathbb{N}}$
of convex nodal functions, namely
$u_h^{k+1}(x_i)\ge u_h^k(x_i)$ for all $x_i\in\Nhi$.
Since this sequence has a uniform upper bound, namely
$u_h^k(x_i) \leq \max_{x_j\in\Nhb} g(x_j)$ for
all $x_i \in \Nhi$ and all $k$, as a consequence of Corollary \ref{MP}
(maximum principle) below, we deduce that the sequence converges
to a nodal function $u_h$.
Next, we show that the limit $u_h$ satisfies \eqref{FEM}.

Taking $\phi(x)$ in Lemma \ref{weakconvergence} (weak convergence of
Monge-Amp\`ere measures) as the hat function
$\phi_i$ associated with $x_i\in\Nhi$ in definition \eqref{fi}, we deduce that
$|\partial u_h^k(x_i)|$ converges and
\begin{align*}
|\partial u_h(x_i)| &=  \int_{\dm}\phi_i(x) d M u_h(x)
\\
& = \lim_{k\to\infty} \int_{\dm}\phi_i(x) d M u_h^k(x)
= \lim_{k\to\infty} | \partial u_h^k(x_i)| \geq f_i,
\end{align*} 
because $Mu_h^k$ is a sum of Dirac measures with mass $|\partial u_h^k(x_i)|$
supported at $x_i\in\Nhi$.
We argue by contradiction. If
$u_h$ does not satisfy \eqref{FEM}, then there exists a node $x_i \in
\Nhi$ such that $|\partial u_h(x_i)| > f_i$ and there exists  $\delta
>0$ such that increasing $u_h(x_i)$ by $\delta$, the subdifferential at
$x_i$ equals $f_i$.
On the other hand, given any $\epsilon>0$,
there exists $k_\epsilon$ such that $0 \le u_h(x_j) - u_h^k(x_j) \leq \epsilon$
for all $k\ge k_\epsilon, x_j\in\Nhi$. 
We define the auxiliary function
$\widetilde{u}^k_h(x_j) := u_h^{k+1}(x_j)$ if $j < i$,
$\widetilde{u}^k_h(x_j) := u_h^{k}(x_j)$ if $j > i$
and $\widetilde{u}^k_h(x_i) := u_h^k(x_i) + \delta -  \epsilon$.
We note that
\[
\widetilde{u}^k_h(x_j) + \epsilon \geq u_h(x_j)
\quad
j \ne i,
\qquad
\widetilde{u}^k_h(x_i) + \epsilon = u_h^k(x_i) + \delta \leq u_h(x_i) + \delta,
\]
whence, applying Lemma \ref{monotonicity}
to $\widetilde{u}^k_h + \epsilon$ and $u_h$ perturbed by $\delta$ at $x_i$
yields
\[
  |\partial \widetilde{u}^k_h(x_i)| \geq f_i.
\]
Therefore,
since $u_h^{k+1}(x_i)$ is the largest value satisfying
$|\partial u_h^{k+1}(x_i)| = f_i$ we deduce
$u_h^{k+1}(x_i) \geq \widetilde{u}^k_h(x_i)
\ge u_h(x_i) + \delta- 2\epsilon > u_h(x_i)$ provided that $\epsilon < \delta/2$.
This contradicts the fact that $u_h^{k+1}(x_i)\le u_h(x_i)$
and proves \eqref{FEM}.

\medskip\noindent
{\bf Computation of subdifferentials.}
Computing $|\partial u_h^k(x_i)|$ is a key step of the algorithm,
which reduces to computing the constant gradients $\nabla\Gamma(u_h^k)|_T$
of the convex envelope $\Gamma(u_h^k)$ of $u_h^k$ for each element
$T$ of the induced mesh $\mathcal{T}_h^k$.
Lemma \ref{char_subdifferential} yields
\[
|\partial u_h^k(x_i)| = \text{measure of polygon with vertices
$\{\gradv \Gamma(u_h^k)|_T\}$ and $x_i\in T$}.
\]
During the iteration the underlying mesh $\mathcal{T}_h^k$ changes,
starting from the Delaunay mesh $\T_h^0$ for $\Nh$,
first for $d=2$ and next for $d=3$. We describe now how these
changes occur and can be implemented.

{\bf Case $d=2$.}
We consider two triangles $T_1,T_2\in\mathcal{T}_h^k$
with vertices $z_1, z_2, z_4$ and $z_2, z_3, z_4$, namely
they are the convex hulls
\[
T_1 = \hull \{z_1, z_2, z_4\},
\quad
T_2 = \hull \{z_2, z_3, z_4\},
\]
and
\[
z_1 = (h, 0),
\quad
z_2 = (0, -h),
\quad
z_3 = (-h, 0),
\quad
z_4 = (0, h).
\]
For $t \in [-1,1]$, we consider the one-parameter convex nodal function
\[
u_h^k(z_i) = 0 \quad i = 1, 2, 3,
\qquad
u_h^k(z_4) = t,
\]
and observe that its Lagrange interpolant
$I_h^k(u_h^k)$ is convex for $-1\le t \le 0$ and concave for $0 \le t \le 1$;
hence $\Gamma(u_h^k) = I_h^k(u_h^k)$ only for $-1\le t \le 0$.
The constant gradients in $T_i$ are
$t\nabla\phi_4$ for $i=1,2$ and the jump on the edge $F_1=[z_2,z_4]$ is given by
\[
\jump{\gradv I_h^k(u_h^k)}|_{F_1} = - \frac{t}{h}.
\]
We see that it is non-negative only for $-1\le t\le 0$, which means
that $I_h^k(u_h^k)$ is convex according to \eqref{sign}. When $t=0$ the
function $I_h^k(u_h^k)$ is linear in $T_1 \cup T_2$ and for $0\le t \le1$
we have to flip the edge $F_1$ to $F_2=[z_1,z_3]$ and consider a new
mesh $\T_h^{k+1}$ upon replacing $T_1,T_2$ with the two new triangles
\[
T_3 = \hull \{z_1,z_2,z_3\},
\quad
T_4 = \hull \{z_1,z_3,z_4\},
\]
for which $\jump{\gradv I_h^{k+1}(u_h^k)}|_{F_2} = \frac{t}{h} \ge 0$ and
$I_h^{k+1}(u_h^k)$ is convex. This example reveals that the connectivity of
the underlying mesh $\mathcal{T}_h^k$ may change as we increase nodal
values $u_h^k(x_i)$. Since
increasing $u_h^k(x_i)$ is equivalent to adding a multiple of $\phi_i^k$
to $u_h^k$, we realize that changes are local and restricted to the
star $\omega_i^k$. They can be monitored on edges within $\omega_i^k$
by simply checking the signs of jumps and flipping edges accordingly.
Jumps on the boundary edges of
$\omega_i$ increase, which is consistent with Lemma \ref{monotonicity},
and require no attention whatsoever. We further see that
the edge flipping process is similar to the construction of Delaunay
meshes for $d=2$.

{\bf Case $d=3$.}
The change of mesh connectivity is more complicated for $d=3$, but
it is still local (within the star $\omega_i^k$) and thus trackable
\cite{Joe91}. To describe this process,
we consider the following setting with five nodes
\[
z_0 = (0,0,-1), \,
z_1 = (1,0,0), \,
z_2 = (1,1,0), \,
z_3 = (0,1,0), \,
z_4 = (0,0,1),
\]
and two configurations for the convex hull of $\{z_0, z_1, z_2, z_3, z_4\}$;
see Figure \ref{fig:twoconfig}. 
The first configuration has two tetrahedra $T_1,T_2$ with one common face $F$
\[
T_1 = \hull \{ z_0, z_1, z_2, z_4\}
\quad
T_2 = \hull \{ z_0, z_2, z_3, z_4\}
\quad
F = \hull \{z_0,z_2,z_4\}.
\]
The second configuration has three tetrahedra $T_1^*,T_2^*,T_3^*$ and three
common faces $F_1^*,F_2^*,F_3^*$
\begin{gather*}
T_1^* = \hull \{ z_4, z_1, z_2, z_3\},
\,
T_2^* = \hull \{ z_0, z_1, z_2, z_3\},
\,
T_3^* = \hull \{ z_0, z_1, z_3, z_4\},
\\
F_1^* = \hull \{z_1, z_2, z_3 \},
\,
F_2^* = \hull \{z_0, z_1, z_3 \},
\,
F_3^* = \hull \{z_1, z_3, z_4 \},
\end{gather*}
whence $F_1^* = T_1^*\cap T_2^*, F_2^* = T_2^* \cap T_3^*, F_3^* = T_3^*\cap T_1^*$.
We will see that perturbing the values of a convex nodal function $u_h^k$,
one configuration switches to the other to keep convexity of $I_h^k (u_h^k)$.

\begin{figure}
\tdplotsetmaincoords{70}{165}
\subfloat[]
{
  \begin{tikzpicture}[scale=2,tdplot_main_coords]
    \coordinate (O) at (0,0,0);


    \tdplotsetcoord{z1}{1}{90}{0}    
    \tdplotsetcoord{z2}{1}{90}{90}   
    \tdplotsetcoord{z3}{1}{90}{180}  
    \tdplotsetcoord{D}{1}{90}{270}  
    \tdplotsetcoord{z4}{2}{30}{270}     
    \tdplotsetcoord{z0}{2}{150}{270}   
    \node at (z0) [right] {$z_0$};
    \node at (z1) [above left] {$z_1$};
    \node at (z2) [above] {$z_2$};
    \node at (z3) [above right] {$z_3$};
    \node at (z4) [above right] {$z_4$};
    \draw (z1) -- (z2) -- (z3);
    \draw (z4) -- (z1) -- (z0);
    \draw (z4) -- (z2) -- (z0);
    \draw (z4) -- (z3) -- (z0);
    \draw[dashed] (D) -- (z2);
    \draw[dashed] (z3) -- (D) -- (z1);
    \draw[dashed](z4) -- (D) -- (z0);
  \end{tikzpicture}
}
\subfloat[]
{
  \begin{tikzpicture}[scale=2,tdplot_main_coords]
    \coordinate (O) at (0,0,0);


    \tdplotsetcoord{z1}{1}{90}{0}    
    \tdplotsetcoord{z2}{1}{90}{90}   
    \tdplotsetcoord{z3}{1}{90}{180}  
    \tdplotsetcoord{D}{1}{90}{270}  
    \tdplotsetcoord{z4}{2}{30}{270}     
    \tdplotsetcoord{z0}{2}{150}{270}   
    \node at (z0) [right] {$z_0$};
    \node at (z1) [above left] {$z_1$};
    \node at (z2) [above] {$z_2$};
    \node at (z3) [above right] {$z_3$};
    \node at (z4) [above right] {$z_4$};
    \draw (z1) -- (z2) -- (z3);
    \draw (z4) -- (z1) -- (z0);
    \draw (z4) -- (z2) -- (z0);
    \draw (z4) -- (z3) -- (z0);
    \draw[dashed] (z1) -- (z3);
    \draw[dashed] (z3) -- (D) -- (z1);
    \draw[dashed](z4) -- (D) -- (z0);
  \end{tikzpicture}
}
\caption{Two conforming partitions of the convex hull of
$z_0, z_1, z_2, z_3, z_4$: the first configuration
contains two tetrahedra $T_1 = \hull \{ z_0, z_1, z_2, z_4\}$ and
 $T_2 = \hull \{ z_0, z_2, z_3, z_4\}$, whereas the second one consists of
three tetrahedra
$T_1^* = \hull \{ z_4, z_1, z_2, z_3\}$, $T_2^* = \hull \{ z_0, z_1, z_2, z_3\}$,
and $T_3^* = \hull \{ z_0, z_1, z_3, z_4\}$.}
\label{fig:twoconfig}
\end{figure}
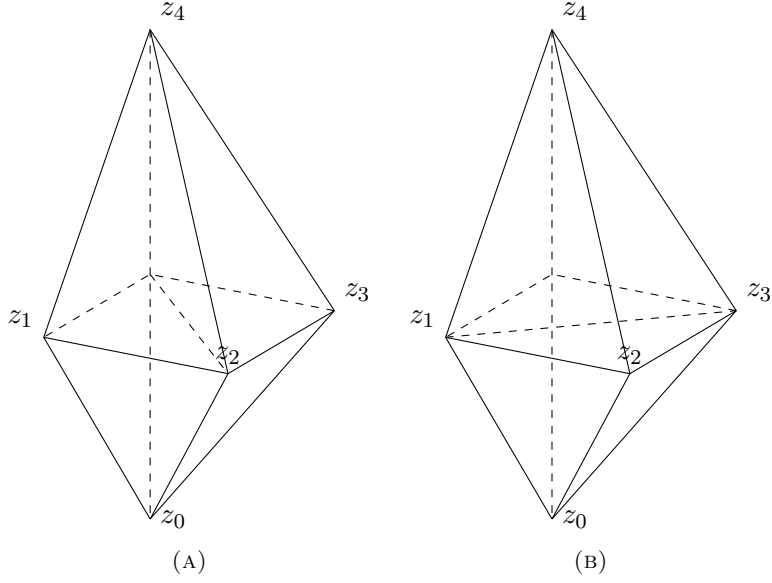

{\it Case I.} We first describe a transition from the first to the second
configuration.
Let $u_h^k$ be the following nodal function 
\[
u_h^k(z_0) = 0, \quad u_h^k(z_1) = u_h^k(z_2) = u_h^k(z_3) = 1, \quad u_h^k(z_4) = 2 +t.
\]
A simple computation yields 
\[
\gradv I_h^k(u_h^k)|_{T_1} = (0,0,1) + (-1, 0, 1) \frac t 2,
\quad
\gradv I_h^k(u_h^k)|_{T_2} = (0,0,1) + (0, -1, 1) \frac t 2.
\]
Since the unit normal vector to $F$ is
$\vf n|_{T_1} = - \vf n|_{T_2} = \frac {\sqrt 2}{2}(-1, 1, 0)$, the jump on $F$ reads
$\jump{\gradv I_h^k(u_h^k)}|_F = -\frac {\sqrt 2}{2} t$ according to \eqref{jump-def},
and changes sign as $t$ increases from $-1$ to $1$.
To preserve the convexity of $I_h^k(u_h^k)$ for $-1\le t<0$,
we switch to the second configuration for $0<t\le 1$. We thus get gradients
\begin{align*}
\gradv I_h^k(u_h^k)|_{T_1^*} = &\; (0,0,1) + (0,0,1) t ,
\\
\gradv I_h^k(u_h^k)|_{T_2^*} = &\; (0,0,1) ,
\\
\gradv I_h^k(u_h^k)|_{T_3^*} = &\; (0,0,1) + (-1,-1,1)\frac{t}{2},
\end{align*}
unit normals
\[
\vf n_1|_{F_1^*}=(0,0,-1), ~
\vf n_2|_{F_2^*}=\frac{1}{\sqrt{3}} (-1,-1,1), ~
\vf n_3|_{F_3^*}=\frac{1}{\sqrt{3}}(1,1,1),
\]
and jumps
\[
\jump{\gradv I_h^k(u_h^k)}|_{F_1^*} = t,
\quad
\jump{\gradv I_h^k(u_h^k)}|_{F_2^*} = \frac{\sqrt{3}}{2} t,
\quad
\jump{\gradv I_h^k(u_h^k)}|_{F_3^*} = \frac{\sqrt{3}}{2} t.
\]
Therefore, the function $I_h^k(u_h^k)$ is convex in
$T_1 \cup T_2$ for $t<0$ and
$T_1^* \cup T_2^* \cup T_3^*$ for $t>0$. The function $I_h^k(u_h^k)$
is linear for $t=0$.

{\it Case II.}
We next describe a transition from the second to the first configuration
upon increasing one nodal value. If
\[
u_h(z_0) = 0,
\quad
u_h(z_1) = u_h(z_2) = 2,
\quad 
u_h(z_3) = 2+t,
\quad 
u_h(z_4) = 4,
\] 
then it is easy to check that 
\begin{align*}
\gradv I_h^k(u_h^k) |_{T_1^*} =&\; (0,0, 2) + (-1, 0, -1) t
\\
\gradv I_h^k(u_h^k) |_{T_2^*} =&\; (0,0, 2) + (-1, 0, 1) t
\\
\gradv I_h^k(u_h^k) |_{T_3^*} =&\; (0,0, 2) + (0, 1, 0) t
\end{align*}
and 
\[
\jump{\gradv I_h^k(u_h^k)} |_{F_1^*} = -2t,
~
\jump{\gradv I_h^k(u_h^k)} |_{F_2^*} = -\sqrt 3 t,
~
\jump{\gradv I_h^k(u_h^k)} |_{F_3^*} = -\sqrt 3 t.
\]
Therefore, $I_h^k(u_h^k)$ is convex in $T_1^* \cup T_2^* \cup T_3^*$ for $t<0$
but not for $t>0$. On the other hand, we see that
\begin{align*}
\gradv I_h^k(u_h^k)|_{T_1} = (0,0,2),
\quad
\gradv I_h^k(u_h^k)|_{T_2} = (0,0,2) + (-1,1,0)t,
\end{align*}
and the jump of the interelement face $F$ is
\begin{align*}
\jump{\gradv I_h^k(u_h^k)} |_F = \sqrt{2}t ,  
\end{align*}
whence $I_h^k(u_h^k)$ is convex in $T_1 \cup T_2$ for $t>0$.
This concludes the discussion.

In the rest of the paper,  we focus on estimating the rates of
convergence. Such an estimate relies on the stability and consistency
of \eqref{FEM}, which we address in sections \ref{S:stability}
and \ref{S:consistency}, respectively.

%
\section{Stability}\label{S:stability}
%

Given two nodal functions $v_h$ and $w_h$ we control the
$L^\infty$-norm of $(v_h-w_h)^-$ in terms of the discrepancy of their
subdifferential measures. This is the content of Proposition \ref{stability}
and its proof is the chief goal of this section.
This result hinges on two important estimates, the
discrete Alexandroff estimate and the Brunn-Minkowski inequality,
which we discuss next. 

\subsection{Discrete Alexandroff Estimate}\label{S:discrete-Alexandroff}
The Alexandroff estimate for a continuous piecewise affine
function $v_h$ states that the $L^{\infty}$-norm of $v_h$ is
controlled by the Lebesgue measure of its subdifferential.
We refer to \cite{NochettoZhang} for a complete proof, 
and also to
\cite{KuoTrudinger90, KuoTrudinger92, KuoTrudinger96, KuoTrudinger00} for
similar estimates and for discrete Alexandroff Bakelman Pucci estimates
for general fully nonlinear elliptic problems.

\begin{lemma}[discrete Alexandroff estimate]\label{L:Alexandroff}
  Let $v_h$ be a nodal function and $v_h(x_i) \geq 0$ at all $x_j \in \Nhb$.
  Then
\begin{align}\label{alex}
  \sup_{\dm} v_h^- \leq C \left( \sum_{x_i \in \mathcal{C}^-_h(v_h) } \abs {\partial v_h (x_i)} \right)^{1/d},
\end{align}
where $C = C(d, \Omega)$ is proportional to the diameter of $\Omega$
and $\mathcal{C}^-_h(v_h)$ is the contact set defined in \eqref{contactset}.
\end{lemma}

\subsection{Brunn-Minkowski Inequality}\label{S:brumm-minkowski}
The second main tool to prove Proposition \ref{stability} is the
celebrated Brunn-Minkowski inequality \cite{Gardner02}.
This inequality relates the Lebesgue measures of compact subsets
$A,B$ of Euclidean space $\mathbb{R}^d$ with that of their Minkowski sum $A+B$
defined in \eqref{minkowski-sum}.
\begin{lemma}[Brunn-Minkowski inequality]\label{BM}
  Let $A$ and $B$ be two nonempty compact subsets of $\mathbb R^d$ for $d \geq 1$. Then the following inequality holds:
\[
  |A + B|^{1/d} \ge |A|^{1/d} + |B|^{1/d}.
\]
\end{lemma}

Since $(a+b)^t \le a^t+b^t$ for $a,b\ge0$ and $0<t<1$, we deduce the
following immediate consequence of Lemma \ref{BM}
\begin{equation}\label{BM-byproduct}
|A + B| \ge |A| + |B|.
\end{equation}

\subsection{Continuous Dependence}\label{S:proof-stab}

We now compare two arbitrary nodal functions in terms of
  their subdifferentials. This is instrumental for the error analysis.

\begin{prop}[continuous dependence]\label{stability}
Let $v_h$ and $w_h$ be two nodal functions associated with nodes $\Nh$ and $v_h \geq w_h$ at all $x_i \in \Nhb$. Then \looseness=-1
\[
  \sup (v_h - w_h)^- \leq C \left( \sum_{x_i \in \mathcal{C}_h^-(v_h - w_h)} \Big( \abs{ \partial v_h (x_i)}^{1/d} - \abs{ \partial w_h (x_i)}^{1/d} \Big )^d  \right)^{1/d},
\]
where $C = C(d, \dm)$ is proportional to the diameter of $\dm$.
\end{prop}

\begin{proof}
Let $v_h,w_h$ be two arbitrary nodal functions.
We consider the convex envelope $\Gamma (v_h - w_h)$ defined in
\eqref{convexenvelope} and the nodal contact set $\mathcal{C}^-_h (v_h- w_h)$
defined in \eqref{contactset}.
Lemma \ref{alex} (discrete Alexandroff estimate) yields 
  \begin{align}\label{stability-alexandroff}
    \sup_{\dm} ( v_h - w_h)^- \le C \left( \sum_{x_i \in \mathcal{C}^-_h(v_h - w_h)} |\partial  \Gamma (v_h - w_h)(x_i) | \right)^{1/d},
  \end{align}
whence we only need to estimate $|\partial  \Gamma (v_h - w_h)(x_i) |$
for all $x_i \in \mathcal{C}^-_h(v_h - w_h)$.
For these nodes, we easily see that
\[
\partial\Gamma(v_h - w_h)(x_i) \subset \partial(v_h - w_h)(x_i).
\]
Consequently, Lemma \ref{Msum} (addition of subdifferentials) gives 
\[
\partial w_h(x_i) + \partial \Gamma (v_h - w_h)(x_i) \subset \partial v_h(x_i)
\quad\forall \, x_i \in \mathcal{C}^-_h(v_h - w_h).
\]
Applying Lemma \ref{BM} (Brunn-Minkowski inequality), we obtain 
\begin{align*}
  |\partial w_h(x_i)|^{1/d} &+ |\partial \Gamma (v_h - w_h)(x_i)|^{1/d} 
  \\
  & \leq
  |\partial w_h(x_i) + \partial \Gamma (v_h - w_h)(x_i) |^{1/d} 
  \leq 
  | \partial v_h(x_i)|^{1/d} ,
\end{align*}
whence
\begin{align*}
  |\partial \Gamma (v_h - w_h)(x_i)| \leq &\; \left(| \partial v_h(x_i)|^{1/d} - |\partial w_h(x_i)|^{1/d}   \right)^{d} .
\end{align*}
This inequality gives us the desired estimate for $ |\partial \Gamma
(v_h - w_h)(x_i)|$. In view of \eqref{stability-alexandroff}, adding
over all $x_i \in \mathcal{C}^-_h (v_h - w_h)$ concludes the proof.
\end{proof}

A direct consequence of this stability result is the maximum principle
for nodal functions, which we state next.
\begin{corollary}[discrete maximum principle]\label{MP}
Let $v_h$ and $w_h$ be two nodal functions associated with nodes $\Nh$.
If $v_h(x_i) \geq w_h(x_i)$ at all $x_i \in \Nhb$ and $|\partial v_h(x_i)| \leq |\partial w_h(x_i)|$ at all $x_i \in \Nhi$, then 
\[
w_h(x_i) \leq v_h(x_i) \quad \forall x_i \in \Nh.
\]
\end{corollary}

\begin{proof}
For any node $x_i\in \mathcal{C}^-_h (v_h - w_h)$, we have
\[
\partial w_h(x_i) \subset \partial v_h(x_i).
\]
Since $|\partial v_h(x_i)| \leq |\partial w_h(x_i)|$ for all $x_i \in\Nhi$, we
deduce $|\partial v_h(x_i)| = |\partial w_h(x_i)|$ for all
$x_i\in \mathcal{C}^-_h (v_h - w_h)$. Consequently, Proposition
\ref{stability} (continuous dependence) implies
\[
\sup (v_h - w_h)^- = 0,
\]
whence $v_h - w_h \geq 0$. This completes the proof.
\end{proof}

\begin{remark}[uniqueness]
If two nodal functions $u_h$ and $w_h$ are both solutions of \eqref{FEM},
then by Corollary \ref{MP} (discrete maximum principle), 
\[
\sup (u_h - w_h)^- = 0 
\quad \mbox{and} \quad 
\sup (w_h - u_h)^- = 0.
\]
Hence, we infer that $u_h = w_h$ at all nodes. 
This shows uniqueness. 
\end{remark}

%
\section{Consistency}\label{S:consistency}
%

In this section, we examine the consistency of the Oliker-Prussner
method \eqref{FEM}. In general, this method is consistent in the sense
that the right hand side of the \eqref{FEM} can be written
equivalently as $\sum_{x_i \in \Nh} f_i \delta_{x_i}$ and this converges to $f$
in measure. 
However, such a concept of convergence is too weak to derive rates of convergence. 
Fortunately, we realize that if internal nodes are translation invariant,
then a reasonable notion of operator consistency holds for 
any convex quadratic polynomial; see Lemma \ref{L:consistency}.
Such property is shown in \cite{BenamouCollinoMirebeau16,Mirebeau15}
for Cartesian nodes. In contrast, we give here an alternative
proof of consistency based on the geometric interpretation of
subdifferentials of convex quadratic polynomials
in the interior of the domain, extend the
results to $C^{2,\alpha}$ and $W^s_p$ functions, and further investigate the
consistency error in the region close to the boundary.

Lemma \ref{char_subdifferential} (characterization of subdifferential)
states that the subdifferential of a convex nodal function $p_h$ at
node $x_i\in\Nhi$ is the convex hull of piecewise constant gradients
of its convex envelope $\Gamma(p_h)$, which in turn is determined
by the nodal values $p_h(x_j)$ in the adjacent set $A_i(p_h)$ of $x_i$
for $p_h$.
The following lemma gives an estimate of the size of $A_i(p_h)$.
\begin{lemma}[size of adjacent sets]\label{estimate_adjacent_set}
Let the nodal set $\Nh$ be quasi-uniform
and shape-regular with constant $\sigma$. 
Let $p$ be a $C^2$ convex function defined in $\dm$. If $ \lambda I
\leq D^2 p \leq \Lambda I$ in $\dm$ for some constants $\lambda, \Lambda>0$
and $p_h:=N_h p$ is the nodal function associated with $p$ defined in
\eqref{nodal-function}, then 
the adjacent set of nodes $A_i(p_h)$ at $x_i$ for $p_h$ 
satisfies
\[
  A_i(p_h) \subset B_{Rh}(x_i) 
\]
where $R = \frac{\Lambda}{\lambda} \sigma^2$
and $B_{Rh}(x_i)$ is the ball centered at $x_i$ with radius $Rh$.
\end{lemma}

\begin{proof}
Let $z$ be an adjacent node of $x_i$ such that 
\[
  |z - x_i| = \max \{ |x_j - x_i|: \, x_j \in A_i(p_h) \} .
\]
Without loss of generality, we may assume that  $x_i = 0$, $p(x_i)=0$
and $\gradv p(x_i) = 0$ and set $x_0 = x_i$.
Let $\omega_0$ be a star at $x_0$ in mesh $\T_h$ associated with nodal set $\Nh$. 
If $z \in \omega_0$, then the assertion is trivial because  $R \geq 1$. 

If $z\notin\omega_0$, then we may assume that
there is a constant $R\ge1$ such that 
$R^{-1} z \in T$ for
some element $T \subset \omega_0$, which implies that $|z| \leq Rh_T$ and
$R^{-1}|z|\ge\rho_T$.
If $\{x_k\}_{k=0}^d$ are the vertices of simplex $T$, then we write
\[
z = R \sum_{k=0}^d \alpha_k x_k,
\qquad
\alpha_k \geq 0,
\qquad
\sum_{k=0}^d \alpha_k =1.
\]
We next note that $p(x_k) \leq \frac 12 \Lambda h_T^2$ for all $1\le k
\le d$ because $D^2 p \leq \Lambda I$ and $|x_k| \leq h_T$. Since $z
\in A_i(p_h)$, there exists a supporting hyperplane $L$ at $x_0$
such that
\[
L(z) = p_h(z),
\qquad
L(x_k) \leq p_h(x_k) \leq \frac 12 \Lambda h_T^2.
\]
Exploiting that $L$ is linear yields
\[
 p_h(z) = R \sum_{k=0}^d \alpha_k L(x_k) \leq \frac 12 \Lambda h_T^2 R
\]
On the other hand, since $D^2 p \geq \lambda I$ and $|z| \geq R\rho_T
= R h_T \sigma_T^{-1} $, we have
\[
p_h(z) = p(z) \geq \frac {\lambda}2 |z|^2 \geq \frac {\lambda}{2} R^2 \sigma_T^{-2} h_T^2.
\]
Combining the last two inequalities implies
\[
  R \leq \frac{\Lambda}{\lambda}\sigma_T^2 \leq \frac{\Lambda}{\lambda}\sigma^2.
\]
This completes the proof.
\end{proof}

Lemma \ref{estimate_adjacent_set} (size of adjacent sets) shows
that $\partial u_h(x_i)$ does not depend on values of $u_h$
on the boundary $\partial\Omega$
for nodes $x_i$ such that $\textrm{dist}(x_i, \bdry_h) \geq R h$.
We now consider nodes which are $Rh$ away from $\partial\Omega_h$ and
gather several properties of subdifferentials of convex quadratic
polynomials. \looseness=-1
\begin{lemma}[properties of convex quadratic polynomials]\label{properties}
Let $p$ be a convex quadratic polynomial such that $\lambda I \leq
D^2p \leq \Lambda I$ and $p_h:=N_h p$ be the nodal function defined
in \eqref{nodal-function}.
If $R =\frac{\Lambda}{\lambda}\sigma^2$,
then the following properties hold:
\\
$\bullet$
The subdifferential $\partial p_h(x_i)$ is a non-empty set for all $x_i\in\Nh$.
\\
$\bullet$ 
If the nodal set $\Nhi$ is translation invariant and
${\rm dist}(x_i, \bdry_h) \geq R h$ for $x_i\in\Nhi$, then a uniform
refinement $\NR$ of $\Nh$ satisfies
\[
    \abs{\partial p_h(x_i)} = 2^d \abs{\partial p_{\frac h 2} (x_i) }.
\]
  $\bullet$ 
  If the nodal set $\Nhi$ is translation invariant and
  ${\rm dist}(x_i, \bdry_h) \geq Rh$ and ${\rm
    dist}(x_j, \bdry_h) \geq R h$ for $x_i,x_j\in\Nhi$, then
  $A_i(p_h) = (x_i-x_j) + A_j(p_h)$ and 
  \[
    \abs{\partial p_h(x_i)} = \abs{\partial p_h (x_j) }.
  \]
\end{lemma}
\begin{proof}
  Take $x_i=0$ for simplicity.
  To prove the first assertion, we just observe that if  $L$ is
  the tangent plane touching $p$ from below at $0$, then $L$ is a
  lower supporting hyperplane of $p_h$ at $0$. This implies that $\gradv L =\gradv p(0)$
  is in the subdifferential of $p_h$ at $0$.

To prove the second assertion, we consider the auxiliary polynomial
\[
q^i(x) := p(x) - \gradv p(0) \cdot x - p(0),
\]
obtained by subtracting the tangent plane of $p$ at $0$. 
Since adding an affine function does not change the measure of the
subdifferential, we have $|\partial p_h(0)| = |\partial q^i_h(0)|$.
Using that $q^i$ is homogeneous of degree $2$ yields
\[
q^i(x) = 4 q^i \Big(\frac x 2\Big). 
\]
Since
\[
4 q_{\frac{h}{2}}^i \Big(\frac{x_j}{2}\Big) = 4 q^i \Big(\frac{x_j}{2}\Big)
= q^i(x_j) = q_h^i(x_j) \ge v \cdot x_j = 2v \cdot
\frac{x_j}{2}
\]
for all $x_j\in\Nhi$ and $v\in\partial q_h^i(0)$, we deduce
\[
  \partial q_h^i(0) = 2 \partial q_{\frac h 2}^i (0),  
\]
whence
$\abs {\partial q_h^i(0)} = 2^d \abs{\partial q_{\frac h 2}^i (0)}$
and the second assertion follows.

To prove the third assertion,
we write $q^i(x) = (x-x_i)^t Q (x-x_i)$ for a suitable positive definite constant matrix $Q$. This implies
\[
q^i(x) = q^j(x+x_j-x_i)
\qquad\forall \, x\in\mathbb{R}^d
\]
along with
\[
A_i(q^i_h) = (x_i-x_j) + A_j(q^j_h),
\qquad
|\partial q^i_h(x_i)| = |\partial q^j_h(x_j)|.
\]
Since $A_i(p_h) = A_i(q^i_h)$ and $\partial p_h(x_i) = \partial q^i_h(x_i)$,
this concludes the proof.
\end{proof}

Now we are ready to prove the consistency of \eqref{FEM}.

\begin{lemma}[consistency of the discrete \MA measure]\label{L:consistency}
Let $p$ be a convex quadratic polynomial such that $\lambda I \leq
D^2p \leq \Lambda I$ and $p_h:=N_h p$ be the corresponding convex nodal
function defined in \eqref{nodal-function}.
Let $\Nhi$ be translation invariant
and $\Th$ be a mesh with nodes $\Nh$ and translation invariant
basis of piecewise linear functions $\{\phi_i\}_{i=1}^n$.
Then
\[
  | \partial p_h(x_i) | = \int_{\dm} \phi_i(x) \det D^2 p(x) dx 
\]
for any point $x_i \in \Nhi$ such that {\rm dist}$(x_i,\bdry_h) \geq Rh$
and $ R= \frac{\Lambda}{\lambda}\sigma^2$. 
\end{lemma}
\begin{proof}
We consider a sequence of uniform refinements $\Tk$ of $\T_h$ and corresponding
nodes $\Nk$ with $h_k = 2^{-k}h$ for $k\ge1$. Let $p_k = N_{h_k}p$ be
the nodal function of $p$ associated with the set $\Nk$ and let
$\gamma_k=\Gamma(p_k)$ be its convex envelope.
Since $\phi_i$ is compactly supported in $\Omega_h\subset\Omega$
and $\gamma_k\to p$ uniformly as $k\to\infty$, Lemma \ref{weakconvergence}
(weak convergence of Monge-Amp\`ere measures) yields
\[
\int_\Omega \phi_i(x) dM \gamma_k (x) \rightarrow
\int_\Omega \phi_i(x) dM p(x)
\quad \text{ as $k \rightarrow \infty$},
\]
or equivalently
\[
\sum_{x_j \in \Nk} \phi_i(x_j) \abs{\partial p_k(x_j)} \rightarrow
\int_{\dm} \phi_i(x) \det D^2 p(x) dx  
  \quad \text{ as $k \rightarrow \infty$. }
\]
Therefore, we only need to prove 
  \[
    \sum_{x_j \in \Nk} \phi_i(x_j) \abs{\partial p_k(x_j)} =  \abs{\partial p_h(x_i)},
  \]
which is independent of $k$.
Since $\int_{\dm} \phi^k_i = 2^{-kd} \int_{\dm} \phi_i$ and
dist$(x_i,\partial\Omega_h)\ge Rh$,
Lemma \ref{properties} (properties of convex quadratic polynomials) leads to
\[
  \abs{\partial p_k(x_j)} = \abs{\partial p_k(x_i)} = \abs{\partial p_h(x_i)}
  \frac{\int_{\dm} \phi^k_i}{\int_{\dm} \phi_i}
  \qquad\forall \, k\ge1,
\]
where $\{\phi_j^k\}$ is the basis of hat functions over $\Tk$.
Consequently, we obtain
\begin{align*}
  \sum_{x_j \in \Nk} \phi_i(x_j) \abs{\partial p_k(x_j)} = 
  \frac {\abs{\partial p_h(x_i)}}{\int_{\dm} \phi_i} \sum_{x_j \in \Nk}
  \phi_i(x_j) \int_{\dm}  \phi_j^k
  =  \abs{\partial p_h(x_i)} 
\end{align*}
because $\int_\dm \phi_i^k = \int_\dm \phi_j^k$ according to
\eqref{meshinvariant} and
$
\sum_{x_j \in \Nk}  \phi_i(x_j)  \phi_j^k  =   \phi_i,
$
or equivalently $\phi_i \in$ span $\{\phi_j^k\}_{x_j\in\N_k}$.
This completes the proof.
\end{proof}

Since $M=\det D^2 p(x)$ is constant for all $x\in\Omega$, Lemma
\ref{L:consistency} (consistency of the discrete \MA measure) implies that
$\int_\Omega \phi_i(x) dx = M^{-1} |\partial p_h(x_i)|$ for $x_i\in\Nhi$
is independent of the mesh $\Th$ supporting the translation invariant basis
$\{\phi_i\}_{i=1}^n$, and so is the notion of consistency in Lemma \ref{L:consistency}.
This is critical because both $\Th$ and $\{\phi_i\}_{i=1}^n$ might not be unique.
The following two local and quantitative consistency estimates
of Proposition \ref{P:Holder-interior} are a consequence of 
Lemma \ref{L:consistency} for nodes away from $\partial\Omega$.

\begin{prop}[interior consistency]
\label{P:Holder-interior}
Let $\Nhi$ be a translation invariant set of nodes, $\Th$ be a mesh with nodes $\Nh$
and translation invariant basis of piecewise linear functions
$\{\phi_i\}$.
Let $u \in C^{2, \alpha}(\overline{B_i})$ be a convex function so that
$\lambda I \leq D^2 u \leq \Lambda I$ in the
ball $B_i:=B_{Rh}(x_i)$ centered at node $x_i\in\Nhi$ and radius $Rh$
with $R =\frac{\Lambda}{\lambda}\sigma^2 $.
If $x_i\in\Nhi$ satisfies ${\rm dist}(x_i, \bdry_h) \geq Rh$, then 
\[
\left| \abs{\partial N_h u (x_i)} - \int_{\dm} \phi_i(x)\det
D^2 u(x) dx \right| \leq C h^{\alpha} |u|_{C^{2,\alpha}(\overline{B_i})}
\int_{\dm} \phi_i(x) dx,
\]
where $C=C(d,\lambda,\Lambda)$ and
$N_h u$ denotes the convex nodal function associated with $u$ defined
in \eqref{nodal-function}. If instead $u \in W^s_q(B_i)$ with
$s-\frac{d}{q}>2$, $s\le 3$,
then there is again $C=C(d,\lambda,\Lambda)$ such that
\[
 \left| \abs{\partial N_h u (x_i)} - \int_{\dm} \phi_i(x)\det
 D^2 u(x) dx \right| \leq C h^{s-2-\frac{d}{q}} |u|_{W^s_q(B_i)} \int_{\dm} \phi_i(x) dx.
\]
\end{prop}
\begin{proof}
To prove the first estimate
we only need to show the inequality
\begin{align*}
| \partial N_h u (x_i) | 
\leq \int_{\dm} \phi_i(x) \det D^2 u(x) dx + Ch^{\alpha}
|u|_{C^{2,\alpha}(\overline{B}_i)} \int_\Omega \phi_i(x) dx,
\end{align*}
because the reverse inequality can be derived similarly.

  Since $u \in C^{2, \alpha}(\overline{B}_i)$, we estimate $u$ by a
  quadratic polynomial $p$ so that 
  \[
    u(x) \leq  p(x) \quad \forall x \in B_{Rh}(x_i),
  \]
  where $p(x_i) = u(x_i), \gradv p(x_i) = \gradv u(x_i)$ and
  $D^2 p = D^2 u(x_i) + C h^\alpha |u|_{C^{2,\alpha}(\overline{B}_i)} I$
  with universal constant $C$.
  If $p_h=N_h p$, then
  Lemma \ref{monotonicity} (subdifferential monotonicity) yields
  \[
  |\partial N_h u(x_i)| \leq |\partial p_h(x_i)|.
  \]
  If $\phi_i$ is the hat function over $\Th$ associated with $x_i$,
  it remains to show 
\[
  |\partial p_h(x_i)| \leq \int_{\dm} \phi_i(x) \det D^2 u(x) dx +
  Ch^{\alpha} |u|_{C^{2,\alpha}(\overline{B_i})}\int_\Omega \phi_i(x) dx.
\]
  Since $(\lambda + Ch^{\alpha})I \leq D^2 p \leq (\Lambda + Ch^{\alpha})I$ and
  \[
    \frac{\Lambda + Ch^{\alpha}} {\lambda + Ch^{\alpha}} \leq  \frac {\Lambda}{\lambda} \quad \text{because $\Lambda \geq \lambda$},
  \]
invoking Lemma \ref{L:consistency} (consistency of the discrete
\MA measure), we obtain
\[
| \partial p_h(x_i) | = \int_{\dm} \phi_i(x) \det D^2 p(x) dx 
\]
because this holds for any mesh $\Th$ with nodes $\Nh$ and translation
invariant basis $\{\phi_i\}_{i=1}^n$
provided dist$(x_i,\partial\dm_h) \ge Rh$.
Recalling that $u\in C^{2,\alpha}(\overline{B_i})$, we can write
$D^2 p = D^2 u(x) + E(x)$ for all $x\in \overline{B_i}$, where
$|E(x)| \le C |u|_{C^{2,\alpha}(\overline{B_i})} h^\alpha$.
Writing $\det D^2p = \det D^2 u(x) \det(I + E(x) D^2 u(x)^{-1})$
and using Taylor expansion yields
\[
| \partial p_h(x_i) |
\leq \int_{\dm} \phi_i(x) \det D^2 u(x) dx + Ch^{\alpha}
|u|_{C^{2,\alpha}(\overline{B_i})} \int_\Omega\phi_i(x) dx ,
\]
and concludes the proof of the H\"older estimate. Finally, if
$u \in W^s_q(B_i)$ with $s-\frac{d}{q}>2$, then we resort to the Sobolev
embedding $W^s_q(B_i) \subset C^{2,\alpha}(\overline{B_i})$ with
$0<\alpha=s-2-d/q<1$ and apply the preceding H\"older estimate.
\end{proof}

For nodes close to the boundary, we can no longer exploit the node
translation invariance and we thus get an error of order $1$.
We express this fact as follows.

\begin{lemma}[boundary consistency]\label{estimate_bdry}
Let $\Th$ be a mesh with nodes $\Nh$ and $\{\phi_i\}$ be a basis of
piecewise linear hat functions over $\Th$.
Let $u \in W^2_\infty(B_i)$ be a convex function with
$\lambda I \leq D^2 u \leq \Lambda I$ in the set
$B_i := B_{Rh}(x_i) \cap \dm$ with $R =\frac{\Lambda}{\lambda}\sigma^2 $,
and let $N_h u$ be the convex nodal function associated with $u$.
If $x_i\in\Nhi$ is a node with ${\rm dist}(x_i, \bdry_h) \leq Rh$, then 
\begin{align*}
 \left| \abs{\partial N_h u (x_i)} - \int_{\dm} \phi_i(x) \det D^2 u(x) dx \right|
 \leq C |u|_{W^2_\infty(B_i)}^d\int_\Omega \phi_i(x) dx,
\end{align*}
where the constant $C=C(d,\lambda,\Lambda)$.
This estimate is valid for any $x_i\in\Nhi$.
\end{lemma}
\begin{proof}
We proceed as in Proposition \ref{P:Holder-interior} and only show 
\[
| \partial N_h u(x_i) | \leq \int_{\dm}  \phi_i(x) \det D^2 u(x)dx
+ C |u|_{W^2_\infty(B_i)}^d \int_\Omega \phi_i(x) dx.
\]
Since $\lambda I \leq D^2 u \leq \Lambda I$, we have
$A_i(N_h u) \subset B_{Rh}(x_i) \cap \dm $ according to
Lemma \ref{estimate_adjacent_set} (size of adjacent sets).
The $W^2_\infty$-regularity assumption of $u$ gives the following estimate
for the piecewise constant gradient of the convex envelope
$\Gamma(N_h u)$ of $N_h u$
over each element $T$ of the mesh induced by $\Gamma(N_h u)$, not
necessarily $\Th$, and contained in $B_{Rh}(x_i)$
\[
\gradv \Gamma(N_h u)|_T = \gradv u(x_i) + v_T,
\qquad 
|v_T| \le C h |u|_{W^2_\infty(B_i)}.
\]
Applying Lemma \ref{char_subdifferential} (characterization of
subdifferential) we deduce that the convex hull of all $\gradv \Gamma(N_h u)|_T$ 
for $T \ni x_i$, whence $\partial N_h u(x_i)$,
can be bounded by a ball of radius $C h |u|_{W^2_\infty(B_i)}$
centered at $\nabla u(x_i)$. Hence, we arrive at
\begin{align*}
  | \partial N_h u(x_i) | &\leq C |u|_{W^2_\infty(B_i)}^d
  \int_{\dm} \phi_i(x) dx
  \\
  & \leq \int_{\dm}  \phi_i(x) \det D^2 u(x) dx +  C |u|_{W^2_\infty(B_i)}^d
  \int_{\dm} \phi_i(x) dx,
\end{align*}
because $\det D^2u(x)\ge0$ a.e. $x\in\dm$. This completes the proof.
\end{proof}

%
\section{Rates of convergence}\label{S:Rates}
%
			   
Our goal in this section is to establish rates of convergence in the
max-norm for the approximation \eqref{FEM} of the 
\MA equation \eqref{pde}.
We first deal with classical solutions $u\in C^2(\overline{\Omega})$ and next with
non-classical solutions
$u\in C^{1,1}(\overline{\Omega})\setminus C^2(\overline{\Omega})$.
Interior error estimates result from combining the stability and consistency estimates
derived in Sections \ref{S:stability} and \ref{S:consistency}.
Boundary error estimates entail a different approach involving
discrete barrier functions, which we discuss next.

%
\subsection{Discrete Barrier function}\label{S:barrier}
%

Since the consistency estimates of Proposition \ref{P:Holder-interior}
are valid in the interior, we need to treat
the boundary layer
\[
\{x\in\Omega: \textrm{dist} (x,\partial\Omega_h) \le Rh\}
\]
differently. We exploit that $N_h u - u_h=0$ on $\partial\Omega_h$
together with the fact that $N_h u - u_h$ cannot grow too fast from
$\partial\Omega_h$. This is a consequence of the next result.

\begin{lemma}[discrete barrier]\label{barrier}
Let $\dm$ be uniformly convex and $\Nhi$ be translation invariant. 
Given a constant $E > 0$, for each node $x_i \in \Nhi$ with
${\rm dist} (x_i , \bdry_h) \leq R h$,
$R=\frac{\Lambda}{\lambda}\sigma^2$,
there exists a nodal function $b^i_h$ such that 
$
\abs{ \partial b^i_h (x_j) } \geq E  \int_{\dm} \phi_j(x) \; dx   
$
for all $x_j \in \Nhi$, $b^i_h(x_j) \leq 0$ at $x_j \in \Nhb$ and 
\[
  \abs{ b^i_h (x_i) } \leq C R E^{1/d} h ,
\]
provided that $h$ is sufficiently small.
\end{lemma}

\begin{proof}
We proceed in three steps. We denote $z = x_i$ for convenience.

{\bf Step 1.} We first construct a nodal function $p_h$ such that 
  \[
  \abs{\partial p_h (x_j)} \geq E \int_{\dm} \phi_j(x) dx
  \quad\forall \, x_j\in\Nhi.
  \]
Let $z_0 \in \bdry$ be such that $|z - z_0| = {\rm dist}(z, \bdry)$. 
We introduce a coordinate system with origin at $z_0$ and $z = (0,
\cdots, 0, |z - z_0|)$ and the domain $\dm$ lying within the sphere
$S_r$ given by
\[
  x_1^2 + x_2^2 + \ldots + x_{d-1}^2 + ( x_d - r )^2 \le r^2, 
\]
where the radius $r$ is a lower bound for the curvature of the
boundary $\partial \dm$ which is strictly positive.
Let $p(x)$ be the convex quadratic polynomial
\begin{align*}
  p(x) =  \frac{E^{1/d}}{2} \left \{ x_1^2 + x_2^2 + \ldots + x_{d-1}^2 + ( x_d - r )^2 - r^2 \right \},
\end{align*}
which is $\le 0$ in $\Omega$.
We consider a extension $\mathcal{N}_h^+$ of the nodal set $\Nhi$ to $\mathbb R^d$
\[
  \mathcal{N}_h^+ = \left\{ z = h \sum_{j=1}^d k_j e_j: k_j \in \mathbb Z \right\}
\]
where $\{e_j\}$ is the basis spanning the translation invariant set $\Nhi$. 
Let $p_h^+=N_h p$ be the nodal function associated with $p$ over $\mathcal{N}_h^+$,
namely $p_h^+(x_j) = p(x_j)$ for all $x_j \in \mathcal{N}_h^+$.
Since $\mathcal{N}_h^+$ is translation invariant,
Lemma \ref{L:consistency} (consistency of the discrete \MA measure) yields
\[
\abs{\partial p_h^+ (x_j)} = \int_{\omega_j} \phi_j(x) \det D^2 p(x) dx
= E \int_{\omega_j} \phi_j(x) dx 
  \quad \forall x_j \in \mathcal{N}_h^+,
\]
where $\omega_j = \textrm{supp} (\phi_j)$.
To define the nodal function $p_h$ on $\Nh$, we set 
\[
p_h(x_j) := p_h^+(x_j) \quad \forall x_j \in \Nhi.
\]
To define $p_h$ at boundary nodes $x_j \in \Nhb$, which may not
belong to $\mathcal{N}_h^+$, we regard the
convex envelope $\Gamma (p_h^+)$ of $p_h^+$ in $\mathbb{R}^d$ as a
natural extension and we assign $p_h(x_j) := \Gamma (p_h^+)(x_j)$ for
all $x_j \in \Nhb$. In view of Lemma \ref{L:disc-subd}
(discrete subdifferential), we realize
that $\partial p_h(x_j) = \partial \Gamma(p_h^+)(x_j)$ for all $x_j\in\Nhi$.

{\bf Step 2.}
We assert that
\[
p_h (x_j) \leq   C E^{1/d}  h 
\quad \forall x_j \in \Nh.
\]
Since $p_h(x_j) \leq 0$ for all $x_j \in \Nhi$, we only need to
show this for $x_j\in\partial\Omega$. For each such node,
due to the convexity of $p_h^+$, we have
\[
  p_h(x_j) = \Gamma (p_h^+)(x_j) \leq \max \{ p_h^+(x_k): x_k\in A_j(p_h^+) \},
\]
where $A_j(p_h^+)$ is the adjacent set of $x_j$ for $p_h^+$.
By Lemma \ref{estimate_adjacent_set} (size of adjacent sets),
$A_j(p_h^+)$ is
contained in a ball $B_j=B_{Rh}(x_j)$.
We thus deduce
\begin{align*}
  p_h(x_j) \leq \max \{ p_h^+(x_k): \, x_k \in \dm + B_{Rh}(0) \}.
\end{align*}
Since $ \Omega$ is contained in the ball $S_r$, we infer that
\begin{align*}
  p_h(x_j) \le  E^{1/d} \left\{ (r + R h)^2 - r^2 \right \}
  \leq 3 E^{1/d} r R h 
\end{align*}
for $h$ sufficiently small.

{\bf Step 3.}
Finally, if we set 
\[
b^i_h (x) := p_h (x) - 3 E^{1/d} r R h,
\]
then we have $b^i_h(x_j) \le 0$ for all $x_j\in\partial \Nhb$. Moreover, we have
\begin{align*}
  \abs{b^i_h(z)} = \abs{p(z) -  3 E^{1/d} r R h } 
  \leq C E^{1/d} r R h .
\end{align*}
This completes the proof.
\end{proof}

%
\subsection{Rates of Convergence for Classical Solutions}\label{S:rates-classical}
%

We prove rates of convergence for $C^2$ classical solutions of \eqref{pde},
assuming either H\"older or Sobolev regularity of $D^2 u$.
This is the content of Theorems \ref{T:Holder-regularity} and
\ref{T:Sobolev-regularity}.
\begin{theorem}[rate of convergence for $C^{2,\alpha}$ solutions]
\label{T:Holder-regularity}
Let $\dm$ be uniformly convex and $\Nhi$ be translation invariant.
Let $u$ be the convex solution of the \MA equation \eqref{pde} with 
$
\lambda I \leq D^2 u \leq \Lambda I
$
in $\Omega$ and $u_h$ be the solution of \eqref{FEM}
with right-hand sides $\{f_i\}_{i=1}^n$ defined
over a mesh $\T_h$ with nodes $\Nh$ and
translation invariant basis $\{\phi_i\}_{i=1}^n$.
If $f(x)\ge \lambda_F>0$ for all $x\in\Omega$ and
$
u \in C^{2, \alpha}(\overline{\dm}),
$
then 
\[
  \inftynorm{u - \Gamma(u_h)}  \leq C h^{\alpha}
\]
where the constant $C = C(d, \dm, \lambda, \Lambda, \lambda_F)
\big( |u|_{C^{2, \alpha}(\overline{\dm})} + |u|_{W^2_\infty(\Omega)} \big)$. 
\end{theorem}
\begin{proof}
Let $\T_h$ be the mesh with nodes $\Nh$ where we define the nodal
values $f_i$ according to \eqref{fi}, and let $N_h u$ be the nodal
function associated with $u$.
Lemma \ref{estimate_adjacent_set} (size of adjacent sets) gives
$A_i(N_h u) \subset B_{Rh}(x_i)$ for all $x_i \in \Nhi$.
Classical interpolation theory thus yields \cite{BrennerScott08}
\[
  \norm{u - \Gamma(N_h u)}_{L^{\infty}(\dm_h)} \leq C h^2 |u|_{W^2_\infty(\Omega)}.
\]
Therefore, we only need to prove that
$|( N_h u - u_h)(x_i)| \le C h^\alpha$ for all $x_i\in\Nhi$,
or equivalently the one-sided estimate
\begin{align}\label{thm:goal}
  \sup_{x_i \in \Nh} \big( N_h u - u_h \big)^-(x_i) \leq C h^{\alpha}
\end{align}
because a corresponding
inequality for $( N_h u - u_h)^+$ can be derived similarly.

{\bf Step 1} (boundary estimate). We first show that for all $x_i \in \Nhi$
such that ${\rm dist}(x_i, \bdry_h) \leq R h$ with
$R =\frac{\Lambda}{\lambda}\sigma^2$
\[
  \big(u_h - N_h u\big)(x_i) \leq C |u|_{W^2_\infty(B_i)} h.
\]
Fix $x_i\in\Nhi$ and let $b^i_h$ be the discrete barrier defined in Lemma \ref{barrier}
(discrete barrier) with free parameter $E$. We consider the nodal function $u_h + b^i_h$,
which satisfies
\[
\partial u_h (x_j) + \partial b^i_h (x_j) \subset  \partial (u_h + b^i_h) (x_j)
\]
due to Lemma \ref{Msum} (addition of subdifferentials). Applying
\eqref{BM-byproduct}, which is a consequence of
Lemma \ref{BM} (Brunn-Minkowski inequality), implies
\begin{align*}
      \abs {\partial (u_h + b^i_h) (x_j)} \geq \abs{\partial u_h (x_j)}
      + \abs{ \partial b^i_h (x_j)} .
\end{align*}
Since $\partial u_h(x_j) = f_j$ according to \eqref{FEM}, invoking Lemma
\ref{barrier} (discrete barrier) and Lemma \ref{estimate_bdry}
(boundary consistency) yields
\[
\abs {\partial (u_h + b^i_h) (x_j)} \geq
f_j + E \int_\Omega \phi_j(x)dx \ge |\partial N_h u(x_j)|
\quad\forall \, x_j\in\Nhi,
\]
provided $E\ge C|u|_{W^2_\infty(B_j)}^d$.
Moreover, $b^i_h(x_j) \le0$ and $u_h(x_j)=N_h u(x_j) = g(x_j)$ for all
$x_j\in\Nhb$ imply $u_h + b^i_h \leq N_h u$ on $\bdry_h$, whence
Corollary \ref{MP} (discrete maximum principle) gives
\[
u_h(x_j) + b^i_h(x_j) \leq N_h u (x_j) \quad \text{for all $x_j \in \Nhi$}.
\]
Finally, the estimate for $b^i_h(x_i)$ of Lemma \ref{barrier} (discrete barrier)
yields
\begin{align}\label{boundary-estimate}
u_h(x_i) - N_h u(x_i) \le - b^i_h(x_i) \le C |u|_{W^2_\infty(B_i)} h.
\end{align}

{\bf Step 2} (interior estimate).
We intend to apply Proposition \ref{stability} (continuous dependence)
to the nodal function
$N_h u - u_h$ for all nodes $x_i\in\Nhi$ with ${\rm dist}(x_i, \bdry)\geq Rh$,
for which we need to compare subdifferentials and verify boundary conditions.
To deal with the former, we restate
Proposition \ref{P:Holder-interior} (interior consistency) with the
help of \eqref{pde_a} and \eqref{FEM}, namely
$\int_\Omega \phi_i(x) \det D^2u(x) dx = f_i = |\partial u_h(x_i)|$:
\[
\abs{\partial N_h u(x_i)} \le |\partial u_h(x_i)| +
Ch^\alpha |u|_{C^{2,\alpha}(\overline{\Omega})} \int_\Omega \phi_i(x) dx.
\]
Setting $\epsilon := Ch^\alpha |u|_{C^{2,\alpha}(\overline{\Omega})}\int_\Omega \phi_i(x) dx$,
we readily see that 
\[
\abs{\partial N_h u(x_i)}^{1/d} - |\partial u_h(x_i)|^{1/d}
\le (f_i + \epsilon)^{1/d} - f_i^{1/d}
\]
for $x_i\in\mathcal{C}_h^-(N_h u-u_h)$.
Since the function $\psi(t)=t^{1/d}$ is concave for $t>0$, we deduce
that $\psi(t+\epsilon) - \psi(t) \le d^{-1} t^{1/d-1} \epsilon$,
whence
\[
\abs{\partial N_h u(x_i)}^{1/d} - |\partial u_h(x_i)|^{1/d}
\le C h^\alpha |u|_{C^{2,\alpha}(\overline{\Omega})} f_i^{1/d-1}\int_\Omega \phi_i(x) dx.
\]
Exploiting now the lower bound of $f$, namely $f(x)\ge \lambda_F>0$ for all
$x\in\Omega$, we estimate $f_i$ from below
\[
f_i = \int_\Omega f(x) \phi_i(x) dx \ge \lambda_F \int_\Omega \phi_i(x) dx,
\]
and insert this bound back into the preceding expression to obtain
\[
\Big(\abs{\partial N_h u(x_i)}^{1/d} - |\partial u_h(x_i)|^{1/d}\Big)^d
\le C \lambda_F^{1-d} h^{\alpha d}|u|_{C^{2,\alpha}(\overline{\Omega})}^d
\int_\Omega \phi_i(x) dx.
\]
In order to apply Proposition \ref{stability} (continuous dependence)
to the nodal function
$N_h u - u_h$ it remains to check boundary conditions on the smaller domain 
\[
  \dm_h^0 := \{x \in \dm: {\rm dist}(x, \bdry_h) \geq Rh \},
\]
where the above calculation is valid. Since  
$\big(u_h - N_h u\big)(x_i) \le C|u|_{W^2_\infty(B_i)}h$ for $x_i\in\Nhi\setminus\Omega_h^0$,
according to  \eqref{boundary-estimate}, Proposition \ref{stability} leads to
\[
\sup_{x_i\in\Omega_h^0\cap\Nhi} \big(N_h u-u_h+C|u|_{W^2_\infty(B_i)} h\big)^-(x_i)
\le C \lambda_F^{1/d-1} |\Omega|^{1/d} h^\alpha|u|_{C^{2,\alpha}(\overline{\Omega})},
\]
or equivalently to
\begin{equation}\label{interior-estimate}
\big(u_h-N_h u\big)(x_i) \le C |u|_{W^2_\infty(B_i)} h +
C \lambda_F^{1/d-1} |\Omega|^{1/d} h^\alpha|u|_{C^{2,\alpha}(\overline{\Omega})}
\end{equation}
for all $x_i\in\Omega_h^0\cap\Nhi$.
Combining \eqref{boundary-estimate} with \eqref{interior-estimate}
proves the desired estimate \eqref{thm:goal} and concludes the proof.
\end{proof}

\begin{remark}[non-degeneracy]\label{R:f-positive}
The lower bound $\lambda_F$ in the non-degeneracy assumption $f(x)\ge \lambda_F>0$ of Theorem
\ref{T:Holder-regularity} may be viewed as a stability constant
because $\lambda_F^{1/d-1}$ blows up as $\lambda_F\to0$. If $f$ vanishes somewhere,
then $\lambda_F=0$ and we have a reduced convergence rate $h^{\alpha/d}$
because $(f_i+\epsilon)^{1/d}-f_i^{1/d} \le \epsilon^{1/d}$ for all $f_i\ge0$.
\end{remark}  
\begin{remark}[$C^{2,\alpha}$-regularity]\label{R:regularity}
\rm
It is worth observing that the assumptions on $u$ in Theorem
\ref{T:Holder-regularity} can be verified from assumptions on
$f$. In fact, if $0 < \lambda_F \le f(x) \le \Lambda_F$, then $0
< \lambda \le D^2 u(x) \le \Lambda$ for some constants
$\lambda, \Lambda$ \cite{Caffarelli90b}, \cite[Section 4.1]{Gutierrez01}. Moreover,
if $f(x) \in C^{\alpha}(\overline{\dm})$ and $\Omega$ is of class
$C^{2, \alpha}$, then $u \in C^{2, \alpha}(\overline{\dm})$ \cite{Caffarelli90a}, \cite[Section 4.3]{Gutierrez01}.
\end{remark}
\begin{theorem}[rate of convergence for $W^s_q$ solutions]
\label{T:Sobolev-regularity}
Let $\dm$ be uniformly convex and $\Nhi$ be translation invariant.
Let $u$ be the convex solution of the \MA equation \eqref{pde} with 
$
\lambda I \leq D^2 u \leq \Lambda I
$
in $\Omega$ and $u_h$ be the solution of \eqref{FEM}
with right-hand side $\{f_i\}_{i=1}^n$ defined in \eqref{fi}
over a mesh $\T_h$ with nodes $\Nh$ and
translation invariant basis $\{\phi_i\}_{i=1}^n$.
If $f(x)\ge \lambda_F>0$ for all $x\in\Omega$ and
$
u \in W^s_q(\dm)
$
with $s>2+\frac{d}{q}$, $s\le3$, $d<q\le\infty$, then 
\[
  \inftynorm{u - \Gamma(u_h)}  \leq C h^{s-2}
\]
where the constant $C = C(d, \dm, \lambda, \Lambda, \lambda_F)
\big( |u|_{W^s_q(\dm)} + |u|_{W^2_\infty(\Omega)} \big)$. 
\end{theorem}
\begin{proof}
We just employ the second estimate of Proposition
\ref{P:Holder-interior} (interior consistency) in
Step 2 of the proof of Theorem
\ref{T:Holder-regularity} (rate of convergence for $C^{2,\alpha}$ solutions).
The key difference is the estimate of the sum
\[
S := h^{s-2-\frac{d}{q}} \left(\sum_{x_i\in\Nhi}
|u|_{W^s_q(B_i)}^d \int_\Omega \phi_i(x) dx \right)^{\frac{1}{d}}
\]
before applying Proposition \ref{stability} (continuous dependence);
recall that $B_i = B(x_i,Rh)$. To exploit
the $\ell^q$-summability of $\{|u|_{W^s_q(B_i)}\}_{x_i\in\Nhi}$
we utilize H\"older inequality with exponents $t=q/d$ and
$t^*=q/(q-d)$ to arrive at
\[
S \le  h^{s-2-\frac{d}{q}}
\left(\sum_{x_i\in\Nhi} |u|_{W^s_q(B_i)}^q\right)^{\frac{1}{q}} 
\left(\sum_{x_i\in\Nhi} \Big(\int_\Omega \phi_i\Big)^{\frac{q}{q-d}} \right)^{\frac{q-d}{dq}}.
\]
  We note that the balls $B_i$ have a finite overlapping property. In fact, if $B_i \cap B_j \neq \emptyset$, then $\dist(x_i, x_j)\leq 2Rh$ and the box 
\begin{equation}\label{omega-j}
\omega_j := \Big\{ x = x_j + \sum_{k=1}^d t e_k:  -\frac h2 \leq t \leq \frac h2 \Big\}
\end{equation}
is contained in $B_{Ch}(x_i)$ where $C = 2R+ \frac 12 d$. Since
$|B_{Ch}(x_i)| \leq Ch^d$ and $|\omega_j| \geq c h^d$, with $c=c(\sigma)$,
we deduce that balls $B_{Ch}(x_i)$ contain a fixed number
of nodes $x_j$ and $B_i$ overlaps with a fixed number of $B_j$'s.
Hence
\[
\left(\sum_{x_i\in\Nhi} \Big(\int_\Omega
\phi_i\Big)^{\frac{q}{q-d}} \right)^{\frac{q-d}{dq}} \le C
|\Omega|^{\frac{q-d}{dq}} h^{\frac{d}{q}},
\]
and we deduce that $S\le C h^{s-2} |u|_{W^s_q(\Omega)}$ to conclude the proof.
\end{proof}

We stress that Theorem 6.2 (rate of convergence for $W^s_q$ solutions)
provides a linear rate for solutions $u\in W^3_q(\Omega)$ with $q>d$,
which is much weaker than  the requirement $u\in W^3_\infty(\Omega)$ of Theorem
\ref{T:Holder-regularity} (rate of convergence for
$C^{2,\alpha}$ solutions) for a similar rate. We explore this further below.

\subsection{Rates of Convergence for Piecewise Smooth Solutions}
\label{S:rates-pwsomooth}

We now study the case where the consistency error may be large in a
small region, for instance when the Hessian $D^2u$ jumps across a
surface within $\Omega$, but is small otherwise. We therefore give up
the assumption $u\in C^2(\overline{\Omega})$.
We exploit the structure of the estimate in Proposition \ref{stability}
(continuous dependence),
namely the fact that its right-hand side accumulates in $\ell_d$.

Before stating our result, we introduce the Minkowski-Bouligand
dimension of a subset $\omega$ of $\dm$. Given a translation invariant set
of nodes $\Nh=\{x_i\}$, let $\{\omega_i\}_{x_i \in \Nh}$ be
the translation invariant partition covering $\dm$ defined in \eqref{omega-j}.
Let $m(h)$ be the number of $\omega_i$'s required to cover $\omega$.
We define the (Minkowski-Bouligand or box) dimension of $\omega$ to be
\[
  \dim \omega := - \lim_{h\to0}  \frac{\log m(h)}{\log h}.
\]
For example, it is easy to check that $\partial B_1$,
the discontinuity set of $D^2u$ in example \eqref{example2}, is of dimension one.
In addition, the solution $u$ satisfies
$u\in W^2_\infty(B_2(0)) \setminus C^2(\overline{B_2(0)})$
and $u \in C^3(\overline{B_1(0)}), C^3(\overline{B_2(0)}\setminus B_1(0))$.
The following theorem explores situations such as this one.

\begin{theorem}[convergence rates for piecewise smooth Hessians]\label{T:C11estimate}
Let $\Omega$ be uniformly convex and $\Nh$ be a translation invariant set
of nodes.
Let $u\in W^2_\infty(\Omega)$ be the convex solution of the \MA equation
\eqref{pde} and satisfy $\lambda I \leq D^2u \leq \Lambda I$
in $\Omega$. Let $D^2 u$ be piecewise smooth in the sense that
$D^2u \in W^s_q(\Omega \setminus \omega)$ with $s>2+d/q$, $s\le 3$,
$d<q\le\infty$, and let $\omega$ have box dimension $n < d$.
If $f\ge \lambda_F>0$ in $\Omega$ and $u_h$ is the solution of \eqref{FEM}
with right-hand side $\{f_i\}_{i=1}^n$ defined in \eqref{fi}
over a mesh $\T_h$ with nodes $\Nh$ and
translation invariant basis $\{\phi_i\}_{i=1}^n$,
then
\[
  \inftynorm{u - \Gamma(u_h)} \leq C h^{s-2} \, |u|_{W^s_q(\Omega \setminus \omega)}
  +C h^{\frac{d-n}{d}} \, |u|_{W^2_\infty(\Omega)},
\]
where the constant $C = C(d,\Omega,\omega,\lambda,\Lambda, \lambda_F)$.
\end{theorem}

\begin{proof}
We argue as in Theorems \ref{T:Holder-regularity} and
\ref{T:Sobolev-regularity}. Therefore, we first observe that $u\in
W^2_\infty(\Omega)$ guarantees
\[
\abs{ (N_h u - u_h)(x_i) } \leq C |u|_{W^2_\infty(\Omega)} h
\]
for all $x_i\in\Nhi$ such that dist $(x_i,\partial\Omega_h) < Rh$. We
split the nodes $x_i\in\Nhi$ such that dist $(x_i,\partial\Omega_h) \ge Rh$
into two sets, those that are at distance $Rh$ to $\omega$, denoted by
$\Nhi(\omega)$, and the complement $\Nhi(\Omega\setminus \omega)$.

In order to apply Proposition \ref{stability} (continuous dependence),
we start with the estimate
\[
|\partial N_h u(x_i)|^{1/d} - |\partial u_h(x_i)|^{1/d} \le
(f_i + \epsilon_i)^{1/d} - f_i^{1/d} \le d^{-1} \lambda_F^{1/d-1} \delta_i
\]
and $\delta_i=\epsilon_i \big(\int_\Omega \phi_i\big)^{1/d-1}$, in
conjunction with the expression of $\epsilon_i$ already
derived in the second estimate of
Proposition \ref{P:Holder-interior} (interior consistency),
to arrive at
\[
\delta_i = C h^{s-2-\frac{d}{q}} |u|_{W^s_q(B_i)} \Big(\int_\Omega \phi_i\Big)^{\frac{1}{d}}
\qquad\forall \, x_i\in\Nhi(\Omega\setminus \omega).
\]
We next resort to
the crude bound $(f_i + \epsilon_i)^{1/d} - f_i^{1/d} \le \epsilon_i^{1/d}$,
together with the expression of $\epsilon_i$ derived in Lemma
\ref{estimate_bdry} (boundary consistency),
to write
\[
\epsilon_i = C |u|_{W^2_\infty(B_i)} \Big(\int_\Omega \phi_i\Big)^{\frac{1}{d}}
\qquad\forall \, x_i \in \Nhi(\omega).
\]
We point out that Lemma \ref{estimate_bdry} is valid for any $x_i\in\Nhi$.
Arguing as in Theorem \ref{T:Sobolev-regularity} (rate of convergence
for $W^s_q$ solutions), we thus obtain
\[
\sum_{x_i\in \Nhi(\Omega\setminus \omega)} \delta_i^d \le C h^{(s-2)d} \,
|u|_{W^s_q(\Omega\setminus \omega)}^d,
\]
as well as
\[
\sum_{x_i\in\Nhi(\omega)} \epsilon_i^d
\le C |u|_{W^2_\infty(\Omega)}^d\sum_{x_i\in\Nhi(\omega)} |\omega_i|
\le C h^{d-n} \,|u|_{W^2_\infty(\Omega)}^d
\]
because $\sum_{x_i\in\Nhi(\omega)} |\omega_i| \le C m(h) h^d \le h^{d-n}$
with $n$ being the box dimension of $\omega$.
Applying now Proposition \ref{stability} (continuous dependence)
to $N_h(u)-u_h$ yields
\[
\sup \big(N_h(u)-u_h\big)^-  \le
C  h^{s-2} \, |u|_{W^s_q(\Omega\setminus \omega)}
+ C h^{\frac{d-n}{d}}\, |u|_{W^2_\infty(\Omega)}.
\]
This is the asserted estimate.
\end{proof}

We conclude with a simple application of Theorem \ref{T:C11estimate}
(convergence rates for piecewise smooth Hessians) to
the example \eqref{example2}. Since
$d = 2$, $s=3$, $q=\infty$, and $n = \dim(\partial B_1) = 1$, we deduce
\[
  \norm{u - \Gamma(u_h)}_{L^{\infty}(\dm_h)} \leq C(u) h^{1/2}.
\]
We point out that the singular set $\omega=\partial B_1$
need not be matched by either
the mesh $\T_h$ associated with the translation invariant nodal set $\Nhi$
or the mesh induced by the convex envelope $\Gamma(u_h)$ of $u_h$.

\bibliographystyle{plain}

\def\cprime{$'$}


\end{document}